\documentclass[12pt]{article}
\usepackage{amsmath,amsfonts,amssymb,amsthm,amscd}
\title{   \LARGE   Glivenko--Cantelli classes and $NIP$ formulas \\ 
}

\author{Karim Khanaki\thanks{Partially supported by IPM grant 1400030118}\\Arak University of Technology}

\newtheorem{Theorem}{Theorem}[section]
\newtheorem{Proposition}[Theorem]{Proposition}
\newtheorem{Definition}[Theorem]{Definition}
\newtheorem{Remark}[Theorem]{Remark}
\newtheorem{Lemma}[Theorem]{Lemma}
\newtheorem{Corollary}[Theorem]{Corollary}
\newtheorem{Fact}[Theorem]{Fact}

\newtheorem{Question}[Theorem]{Question}

\newtheorem{Convention}[Theorem]{Convention}

\def\dotminus{\mathbin{\ooalign{\hss\raise1ex\hbox{.}\hss\cr
  \mathsurround=0pt$-$}}}

\begin{document}
\maketitle

\begin{abstract}   We give several new equivalences of $NIP$ for formulas and   new proofs of known results using   \cite{Talagrand} and  \cite{HOR}. We emphasize that Keisler measures are more complicated than types (even in the $NIP$ context), in an analytic sense.	Among other things, we show that for a first order theory $T$ and a formula $\phi(x,y)$, the following are equivalent:

\medskip\noindent
 (i) $\phi$ has $NIP$ with respect to $T$.
 \newline
  (ii)  For any global $\phi$-type $p(x)$ and any   model $M$, if $p$ is finitely satisfiable in $M$, then $p$ is generalized $DBSC$ definable over $M$. In particular, if $M$ is countable, then  $p$ is  $DBSC$ definable over $M$. (Cf. Definition~\ref{Borel definability}, Fact~\ref{strong Borel=DBSC}.)
\newline
(iii)  For any  global Keisler $\phi$-measure $\mu(x)$ and any  model $M$, if $\mu$ is finitely satisfiable in $M$, then $\mu$ is generalized Baire-1/2 definable over $M$.  In particular, if $M$ is countable,  $\mu$ is  Baire-1/2  definable over $M$. (Cf. Definition~\ref{strong Borel measure}.) 
\newline
(iv) For any model $M$ and any Keisler  $\phi$-measure 
$\mu(x)$ over $M$,
\begin{align*}
\sup_{b\in M}\Big|\frac{1}{k}\sum_{i=1}^k\phi(p_i,b)-\mu(\phi(x,b))\Big|\to 0,
\end{align*}
\newline
 for almost every $(p_i)\in S_{\phi}(M)^{\Bbb N}$ with the product measure $\mu^{\Bbb N}$.  (Cf. Theorem~\ref{fim}.) 
 \newline
 (v) Suppose moreover that $T$ is countable and $NIP$, then for any countable model $M$, the space of global $M$-finitely satisfied types/measures is a Rosenthal compactum. (Cf. Theorem~\ref{global angelic}.)
\end{abstract}

\section{Introduction} \label{1}  
This paper is a kind of companion-piece to \cite{KP} and \cite{K3}  although here we study equivalences of $NIP$ formulas and local Keisler measures in classical first order logic.   We  also give   alternatives/refinements  of some  results of \cite{HP}, \cite{HPS} and \cite{Gannon sequential}. The main results/observations are Lemma~\ref{bounded variation-continuous}, Theorem~\ref{DBSC measure}, Theorem~\ref{fim}, 
 Corollaries~\ref{generalized definable}, Theorem~\ref{almost fam}, 
  and Theorem~\ref{global angelic}.

We believe that the novelty of the results in this article is as follows:  assuming $NIP$ for a formula $\phi$, we show that $\phi$-types correspond to a {\em proper} subclass of Baire-1 functions (i.e. $DBSC$) and Keisler $\phi$-measures correspond to another {\em proper} subclass of Baire-1 functions (i.e. Baire-1/2), and these classes are drastically {\em different}, in an analytic sense.
This means that measures are definitely more complicated than types, in this sense. The complexity can lead to the definition of new classes of $NIP$ theories in the classical logic; that is, we can distinguish between theories in which types and measures have the same complexity and those that do not.\footnote{Cf. Definition~\ref{semi-NIP}, Theorem~\ref{DBSC measure}, and subsection ``Thesis" below.}
In the framework of continuous logic \cite{BBHU} this complexity has even more important consequences, as this leads to the fact,  contrary to a claim in the literature,
 that there is {\em no} a perfect analog of Shelah's
theorem, namely a theory is unstable iff it has $IP$ or $SOP$ (cf. \cite{K-classification}). On the other hand, using the Talagrand's characterization of Glivenko-Cantelli classes, we give some new equivalences of $NIP$ for formulas that can have future applications.
Finally, we show  how to transfer these local results  into results for {\em global} types/measures. This transfer will lead to new results and new proofs of the  previous results.

Let us give some historical points on the subject.
The notion of  {\em dependence} for a family of sets/functions was introduced independently in model theory \cite{Sh},   learning theory \cite{VC}, and   function theory \cite{Ros}.\footnote{Recall that in  model theory the notion of dependence and the $NIP$ (non independence property) are the same, and its negation is called the $IP$ (independence property).}  The combinatorial nature of the concept soon became apparent. The connection between the Vapnik-Chervonenkis results in learning theory and the $NIP$ in model theory was discovered by Laskowski \cite{Las}.
In fact,  \cite{VC} characterizes the Glivenko-Cantelli classes  using the notion of independence.
In the function theory setting, the initial results led to the invention of the concept of {\em Fremlin-Talagrand stability} by David~Fremlin (cf. Definition~\ref{Talagrand-stable} below).
The connection of this concept with the Glivenko-Cantelli classes was later identified and studied by Talagrand \cite{Talagrand}. His characterization of the Glivenko-Cantelli classes has some advantages over \cite{VC}, as it is more suitable for studying (real-valued) functions and also has no additional measurability conditions. In fact, all the power of that concept lies in empirical  measures or convexity. The importance of this issue was understood in \cite{HPP}, in which the concept of Keisler measure was studied. Keisler measures had previously been invented and studied in \cite{Keisler-fork}.
Thus all the power of the notion (i.e. convexity and empirical  measures) was used in model theory  in \cite{HPP} and \cite{HP}. In the present paper, the relationship between $NIP$ and the Glivenko-Cantelli classes is explicitly studied and a completely functional expression of the subject is presented.

Let us present the motivation of this article. Naturally, it can be interesting to examine the relationship between one area of mathematics and other areas. When a particular concept is independently introduced and studied  in different areas, that concept is inherently of special mathematical importance. For example, we can refer to the stability property and $NIP$ in model theory, which also appear in various other fields such as functional analysis and learning theory. This allows the tools, results, and approaches of one field to be used in another field, and vice versa. This approach also improves the previous results and leads to new results, and sheds light on what was previously known.

Surprisingly, many tools already existed in functional analysis and probability theory.  Some of our results/observations-- not all them--  are just routine translations of the results of \cite{Talagrand} and \cite{HOR} into the language of model theory. However, there are several reasons that justify and even necessitate the use of this translation: (1) Here we provide a useful dictionary (from model theory to analysis and vice versa) that helps to understand the connections between these areas of mathematics and it can be used in  future applications. (2) Many results in the language of analysis can be more easily generalized to   continuous logic  and  these results have interesting applications both for logic itself and for  mathematics (see for example \cite{K-classification} and \cite{K-Banach}). 

 On the other hand, it is worthwhile to provide  more direct arguments and model-theoretic proofs of the results of this  paper, although we believe that any  model-theoretic proof of these analytical facts will necessarily express a similar argument in the language of   model theory and will not differ much in nature.

This paper is organized as follows. In  section 2, we   provide all function theory background and emphasize the relationship between some of the classes of functions that we will use later.

In  section 3, we present the relationship between types/measures and analytical concepts, and  provide a useful dictionary between model theory and analysis. We also emphasize the distinction between types and measures (over countable models or sets) in an analytic sense.

In  section 4, we study the relationship between Glivenko--Cantelli-classes and $NIP$ formulas,  and using alternative proofs, we provide generalizations of some of the results of section 3. 

In Section~5, we give a short proof (and refinements) of a new result on angelic spaces of {\em global} types/measures.

\section{Preliminary for function spaces}
If $f$ is a real-valued function on a set $X$, the symbol $[f>r]$ stands for $\{x\in X:f(x)>r\}$, and $f^{-1}[F]$ stands for $\{x\in X:f(x)\in F\}$. Similarly, we abbreviate $[f\geq r]$, $[f<r]$ and $[f\leq r]$. We define the {\bf positive} and {\bf negative parts} of $f$ to be 
$$f^+(x)=\max(f(x),0), \ \ \ \ \ \ \ \ f^-(x)=\max(-f(x),0).$$
Note that $f=f^+-f^-$ and $|f|=f^++f^-$.

Let $F\subseteq{\Bbb R}^X$ be a set of real-valued functions on $X$.
The topology
of {\em pointwise convergence} on $F$ is that inherited from the
usual product topology of ${\Bbb R}^X$. A typical neighborhood of
a function $f$ is $$U_f(x_1,\ldots,x_n;\epsilon)=\{g\in F:|f(x_i)-g(x_i)|<\epsilon\text{ for all }i\leq n\},$$ where
$\epsilon>0$ and $\{x_1,\ldots,x_n\}$ is a finite subset of $X$.
In this paper, the topology on subsets of ${\Bbb R}^X$ is
pointwise convergence. Otherwise, we explicitly state what is our
desired topology.

It is easy to verify that a sequence $(f_n)_{n<\omega}$  of real-valued functions on $X$ converges pointwise to a function $f$ if for any $\epsilon>0$ and any $x\in X$ there is a natural number $n_{\epsilon,x}$ such that $|f_n(x)-f(x)|<\epsilon$ for all $n\geqslant n_{\epsilon,x}$.  In this case, sometimes we write  $f_n\to f$. We say that $(f_n)_{n<\omega}$ converges {\em uniformly} to $f$ if  for any $\epsilon>0$  there is a natural number $n_{\epsilon}$ such that $|f_n(x)-f(x)|<\epsilon$ for all $n\geqslant n_{\epsilon}$ and all $x\in X$.

Let  $(f_n)_{n<\omega}$  be a sequence of   real-valued functions on $X$.  We always assume that  $f_0\equiv 0$, that is,  $f_0(x)=0$ for all $x\in X$.  
   
   \begin{Fact} \label{0} Let $(f_n)_{n<\omega}$  be as above  (i.e. $f_0\equiv 0$). Then the following are equivalent:
   	\begin{itemize}
   		\item[(i)] 	There is $C>0$ such that $\sum_{n<\omega}|f_{n+1}(x)-f_n(x)|\leq C$ for all $x\in X$.
   		\item[(ii)] There is a function $f$ such that: (a) $f_n\to f$, and (b) $f=F_1-F_2$ where  $F_1=\sum_{n<\omega}(f_{n+1}-f_n)^+$ and $F_2=\sum_{n<\omega}(f_{n+1}-f_n)^-$, and $F_1,F_2$ are bounded.
   	\end{itemize}
   \end{Fact}
\begin{proof}
     (i)~$\Longrightarrow$~(ii): Note that $f_n=\sum_{k=0}^{n-1}(f_{k+1}-f_k)$ and by (i), for all $x$, the infinite series $\sum_{n<\omega}(f_{n+1}(x)-f_n(x))$ is absolutely convergent. Therefore, the sequence $f_n(x)$ converges to $\sum_{n<\omega}(f_{n+1}(x)-f_n(x))$, for all $x$. Let $f=\lim_n f_n$. Clearly, $f=\lim_nf_n=\lim_nf_n^+-\lim_nf_n^-$. (Note that $\lim_nf_n^+\leq \lim_n|f_n|\leq C$. Similarly for $\lim_nf_n^-$.)
     \newline (ii)~$\Longrightarrow$~(i): Suppose that there is $N>0$ such that for all $x$, $\sum_{n<\omega}(f_{n+1}-f_n)^+(x)<N$ and $\sum_{n<\omega}(f_{n+1}-f_n)^-(x)<N$. Therefore $$\sum_{n<\omega}|f_{n+1}(x)-f_n(x)|=\sum_{n<\omega}(f_{n+1}-f_n)^+(x)+\sum_{n<\omega}(f_{n+1}-f_n)^-(x)<2N.$$
\end{proof}

In the following, we give a definition for the metric spaces, but  we will shortly provide a general definition (i.e. without metrizability). The reason is that, in general, the space of types (Stone space) is not metrizable.\footnote{We study metrizable and non-metrizable separately because they have completely different behavior, so that the first one is almost tame, but the second one is very complicated and some aspects of it are still unknown.}

\begin{Definition} \label{subclasses}  {\em Let $X$ be a compact metric space. The
	following is a list of the spaces of real-valued functions defined
	on $X$. The important ones to note immediately are (iv) and (v).
	
	\noindent (i) $C(X)$ is the space of continuous real-valued
	functions on $X$.
	
		\noindent (ii) $SC(X)$, the space of  lower semi-continuous functions, is the space of functions $f:X\to\Bbb R$ such that for every  $r\in \Bbb R$, $[f\leq r]$ is  closed.

	\noindent (iii) $B_1(X)$, the  Baire 1 class, is the space of
	functions $f:X\to\Bbb R$ such that for every closed $F\subseteq \Bbb R$, $f^{-1}[F]$ is  $G_{\delta}$ in $X$, i.e. it is of the form $\bigcap_{n=1}^\infty G_n$ where the $G_{n}$'s are open.

	\noindent (iv)  $DBSC(X)$  is the space of real-valued functions
	$f$ on $X$ such that there are continuous functions $(f_n)_{n<\omega}$ (with $f_0\equiv 0$) and a constant $C>0$ such that $f_n\to f$ pointwise and 
	$\sum_{n<\omega}|f_{n+1}(x)-f_{n}(x)|<C$ for all $x$.

	\noindent (v)  $B_{1/2}(X)$ or Baire-1/2, is the space of
	real-valued functions $f$  on $X$ such that $f$ is the uniform
	limit of a sequence $(F_n)\subseteq DBSC(X)$.}
\end{Definition}

\begin{Fact} \label{DBSC fact}
 Let $X$ be a compact metric space.

	\noindent (i)  Every $f\in DBSC(X)$   is the difference of bounded
	semi-continuous functions; that is,  there are bounded semi-continuous functions $F_1,F_2$
	with $f=F_1-F_2$. 
	
		\noindent (ii)  Let $f$ be a bounded real-valued function on $X$. $f\in B_1(X)$ if and only if  it is the poinwise limit of a
	sequence of continuous functions.

    \noindent(iii)  A real-valued function $f$ on
  $X$ is lower semi-continuous if and only
if there is a sequence $(f_n)$ of continuous functions such that
$f_1\leq f_2\leq \cdots$ and $(f_n)$ converges pointwise to $f$
(for short we write $f_n\nearrow f$).

		 \noindent (iv)  Let $f$ be a bounded real-valued function on $X$.  If there are bounded semi-continuous functions $F_1,F_2$ with $f=F_1-F_2$, then  $f\in DBSC(X)$. Therefore, by (i),  $f\in DBSC(X)$ if and only if $f$   is the difference of bounded
		 semi-continuous functions.
		 \newline
		 (v)  For any simple\footnote{A function is called {\em simple} if its range is
		 	a finite set.}  function $f$ on $X$,  $f\in B_{1/2}(X)$ if and only if it is $DBSC$.
		 \newline
		 (vi) For uncountable
		 compact metric space $X$,
		 $$C(X)\subsetneqq SD(X)\subsetneqq DBSC(X)\subsetneqq B_{1/2}(X)\subsetneqq  B_1(X).$$
\end{Fact}
\begin{proof}
(i): Set $F_1=\sum\big(f_{n+1}(x)-f_{n}(x)\big)^+$ and 
$F_2=\sum\big(f_{n+1}(x)-f_{n}(x)\big)^-$ and use Fact~\ref{0}. (Note that $F_1,F_2$ are lower semi-continuous and bounded. See also (iii).)

\noindent (ii) This is a classical result due to Lebesgue and Hausdorff (see \cite[Theorem~24.10]{Kechris}.) Note that  the limit of any sequence of continuous functions   is Baire 1 in any topological space. For a generalization of (ii), 
see \cite[p.~ 393]{Kuratowski} and also \cite{Jayne}.

\noindent (iii)  is a classical result in functional analysis, essentially due to Baire (cf.  \cite[p. 274]{H}).

\noindent (iv) follows from (iii). Indeed, suppose that $f=F_1-F_2$ and $F_1,F_2$ are bounded lower semi-continuous. By (iii), there are continuous functions $h_n,g_n$ such that $h_n\nearrow F_1$ and $g_n\nearrow F_2$. Therefore, $f_n:=(h_n-g_n)\to f$ and $\sum|f_{n+1}-f_n|\leqslant \sum|h_{n+1}-h_n|+\sum|g_{n+1}-g_n|<C$ for some $C>0$.
(See also \cite{HOR}.)

	\noindent (v) follows from \cite[Proposition~2.2]{CMR}.
	
	\noindent (vi) follows from  \cite[Proposition~5.1]{HOR} and \cite{CMR}.
\end{proof}

Note that the difference between the given classes $B_{1/2}$ and $DBSC$ is fundamental, and in fact there are many other classes between them (cf. \cite{HOR}).

\begin{Fact} \label{DBSC=strong} Let $X$ be a compact metric space and $f$   a simple function on $X$. $f\in DBSC(X)$ iff it is strongly Borel, that is, there exist  disjoint differences of closed sets $W_1,\ldots,W_n$ and real numbers $r_1,\ldots,r_n$ such that $f=\sum_{i=1}^n r_i\chi_{W_i}$.
\end{Fact}
\begin{proof} This is equivalence (3)~$\Longleftrightarrow$~(4) of \cite[Proposition~2.2]{CMR}.
\end{proof}

We define a norm on $DBSC$.

\begin{Definition} 
	{\em Let $X$ be a compact metric space.
		\newline  Suppose that $f\in DBSC(X)$. The $D$-norm of $f$, denoted by $\|f\|_D$, is defined by $$\|f\|_D=\inf\{C:\exists(f_n)\subseteq C(X), \   f_0\equiv 0, \ f_n\to f \text{ pointwise}, \ \sum|f_{n+1}-f_n|\leq C\}.$$}
\end{Definition}

\begin{Fact}[\cite{HOR}, page 3]   \label{properties}  Let $X$ be a compact metric space.     The classes $DBSC$ and $B_{1/2}$
	 are Banach algebras with respect to the $D$-norm and the uniform norm, respectively.
\end{Fact}

We now generalize Definition~\ref{subclasses} for arbitrary compact spaces.

\begin{Definition} \label{generalized subclasses}  {\em Let $X$ be a compact Hausdorff space (not necessarily metric).

		\noindent (i) A simple function $f$ on $X$ is called   a {\em generalized $DBSC$ function} (or short {\em generalized $DBSC$}) if there exist  disjoint differences of closed sets $W_1,\ldots,W_n$ and real numbers $r_1,\ldots,r_n$ such that $f=\sum_{i=1}^n r_i\chi_{W_i}$.

		\noindent (ii)  A real-valued functions $f$  on $X$ is called   {\em generalized Baire-1/2}  if   $f$ is the uniform
		limit of a sequence $(F_n)$ of simple $DBSC$ functions.}
\end{Definition}

Similar to Definition~\ref{subclasses}(iii) above, in the general case, i.e. when $X$ is not metric space, a function $f:X\to\Bbb R$ is 
called Baire~1 if for every closed $F\subseteq \Bbb R$, $f^{-1}[F]$ is  $G_{\delta}$ in $X$. It is well-known that for non-metric spaces, a Baire~1 function is not necessarily the limit of a sequence of continuous functions. For the same reason, a generalized $DBSC$ (or generalized Baire-1/2) function  is not necessarily the limit of a sequence of continuous functions.
These are topics related to to descriptive set theory.

\section{Types and measures}
In this section we give   model-theoretic results using ideas from functional analysis as presented in the previous section.

We work in the classical  ($\{0,1\}$-valued) model theory context.
 Our model theory notation is standard, and a text such
as \cite{Simon} will be sufficient background for the model
theory part of the paper. 

We fix an $L$-formula $\phi(x,y)$, a complete   $L$-theory $T$, the monster model $\cal U$, and
a subset $A$ of $\cal U$. We let $\tilde\phi(y, x)
= \phi(x, y)$. Let $X=S_{\tilde{\phi}}(A)$ be the space of
complete $\tilde{\phi}$-types on $A$, namely the Stone space of
ultrafilters on the Boolean algebra generated by formulas $\phi(a,y)$
for $a\in A$.  Each formula $\phi(a,y)$ for  $a\in A$ defines a
function $\phi(a,y):X\to\{0,1\}$, which takes $q\in X$ to 1 if
$\phi(a,y)\in q$ and to 0 if $\phi(a,y)\notin q$. Note that $X$
is compact and these functions are {\em continuous}, and as $\phi$
is fixed we can identify this set of functions with $A$. So, $A$
is a subset of all bounded continuous functions on $X$, denoted by
$A\subseteq C(X)$.  Just as we did above, one can define $B_1(X)$, Baire-1/2
and $DBSC(X)$.

\begin{Convention} (i): In  this paper, a type over  $\cal U$   is called a ``global type",   a type over a small set $A$ is called a ``non-global type", a consistent collection of Boolean combinations of the instances $\phi(x,b), b\in\cal U$, is called a ``global $\phi$-type", and a consistent collection of Boolean combinations of the instances $\phi(x,b), b\in A$, is called a ``non-global $\phi$-type."
\newline
(ii): In this paper, whenever $\mu$ is a non-global measure/$\phi$-measure over a small set $B$, we assume that $\mu$ is $A$-invariant for some $A\subseteq B$ such that every type over $A$ is realized in $B$. This makes some definitions consistent. Similarly for non-global  types/$\phi$-types.
\end{Convention}

\begin{Definition}[Keisler measures]
{\em  Let $T$ be a complete theory, $\cal U$ the monster model of $T$, and $A\subseteq\cal U$. \newline (i) A (Keisler) measure $\mu$ over $A$ in the variable $x$ is a finitely additive probability measure on the algebra ${\cal L}(A)$ of $A$-definable sets in the variables $x$.  We will sometimes  write $\mu$ as $\mu(x)$ to emphasize that $\mu$ is a measure
on the variable $x$. We say $\mu$ is a global (Keisler) measure,  if $A=\cal U$, and is  non-global (Keisler) measure, in otherwise. 
 \newline (ii) Let $\phi(x,y)$ be a formula. A (Keisler) $\phi$-measure $\mu$ over $A$ in the variable $x$ is a finitely additive probability measure on the algebra ${\cal L}_\phi(A)$ generated by definable sets $\phi(x,b), b\in A$.
  We will sometimes write $\mu$ as   $\mu_\phi$ (or $\mu_\phi(x)$) to emphasize that $\mu$ is a $\phi$-measure
 (on the variable $x$).  A (Keisler) $\phi$-measure $\mu$ is a global (Keisler) $\phi$-measure,  if $A=\cal U$, and is  non-global (Keisler) $\phi$-measure,  otherwise.}
\end{Definition}

 We let  $S(\cal U)$ be the space of global types, and  ${\frak M}({\cal U})$ the collection  of global Keisler measures. We let  $S_\phi(\cal U)$ be the space of global $\phi$-types, and  ${\frak M}_\phi({\cal U})$ the collection  of global Keisler $\phi$-measures. Similarly, for $A\subseteq \cal U$, the symbols $S(A)$, ${\frak M}({A})$, $S_\phi(A)$ and ${\frak M}_\phi(A)$ are defined.
 
For a formula $\phi(x,y)$, we let $\Delta_\phi$ be the set of all Boolean combinations of the instances of $\phi(x,y)$ (with the variable $x$). A formula  in $\Delta_\phi$ is called a $\phi$-formula. If $\mu\in {\frak M}({A})$, we also define $\mu|_\phi(\theta(x,b)):=\mu(\theta(x,b))$ for all $\theta\in\Delta_\phi$ and $b\in A$. Note that $\mu|_\phi\in {\frak M}_\phi({A})$ and so we will sometimes write $\mu|_\phi$ as $\mu_\phi$ (or $\mu$ again).

\begin{Definition} \label{fs-measure}
	{\em Let $A,B$ be two small sets (of $\cal U$) and $\mu\in{\frak M}(A)$ (resp. $\mu\in{\frak M}_\phi(A)$). We say that $\mu$ is {\em finitely satisfiable} in $B$ if  for every $L(A)$-formula (resp. $\phi$-formula with parameters in $A$) $\theta(x)$  such that $\mu(\theta(x))>0$, there exists $a\in B$ such that $\models \theta(a)$.}
\end{Definition}
\begin{Remark}
Recall that the notion of {\em finite satisfiability} can be expressed in topological terms. Indeed, $p\in S(A)$ is finitely satisfiable in $B$ if there are $(b_i)$ in $B$ such that $tp(b_i/A)\to p$ in the logic topology.
A similar topological representation holds for measures, which will be presented in Fact~\ref{Dirac}(iv),(v) below. 
\end{Remark}

\medskip
In the following, we  give a useful dictionary that is used in the rest of the paper and can also   be used in future work.

\begin{Fact} \label{Dirac} Let $T$ be a complete theory, $M$ a small set/model  and $\phi(x,y)$ a formula.
	\newline
	(i)(Pillay) There is a correspondence between global $M$-finitely satisfiable $\phi$-types $p(x)$ and the functions in the pointwise closure of   all functions $\phi(a,y):S_{\tilde{\phi}}(M)\to \{0,1\}$ for $a\in M$, where  $\phi(a,q)=1$ if and only if $\phi(a,y)\in q$.
	\newline
	(ii) The map $p\mapsto\delta_p$ is a correspondence   between global  $\phi$-types $p(x)$  and  Dirac measures $\delta_p(x)$ on $S_{\phi}({\cal U})$, where  $\delta_p({\cal B})=1$ if $p\in{\cal B}$, and equals 0   otherwise. Moreover,    $p(x)$ is finitely satisfiable in $M$ iff $\delta_p(x)$ is finitely satisfiable in $M$.
	\newline
	(iii) There is a correspondence between global  $\phi$-measures $\mu(x)$ and  regular Borel probability measures on $S_{\phi}({\cal U})$. Moreover, a  global  $\phi$-measures is finitely satisfiable in $M$  iff its corresponding regular Borel probability measure is finitely satisfiale in $M$.
	\newline
	(iv) The closed convex hull of Dirac measures  $\delta(x)$  on $S_{\phi}({\cal U})$ is exactly all regular Borel probability measures $\mu(x)$ on $S_{\phi}({\cal U})$. Moreover,
	 the closed convex hull of Dirac measures on $S_{\phi}({\cal U})$ which are finitely satisfiable in $M$   is exactly all regular Borel probability measures $\mu(x)$ on $S_{\phi}({\cal U})$ which are  finitely satisfiable in $M$. 
	\newline
	(v) There is a correspondence between global $M$-finitely satisfiable $\phi$-measures $\mu(x)$ and the functions in the pointwise closure of   all  functions of the form $\frac{1}{n}\sum_1^n \theta(a_i,\bar y)$ on 
	$S_{\bar y}(M)$, where $\theta\in\Delta_\phi$, $a_i\in M$,  and  $\theta(a_i,q)=1$ if and only if $\theta(a_i,\bar y)\in q$.
\end{Fact}
\begin{proof}
	(i) is due to Pillay (cf. \cite[Remark~2.1]{Pillay}).
	\newline
	(ii): Let $X=S_\phi({\cal U})$. For $p(x)\in X$, set the Dirac's measure $\delta_p(x)$  on $X$, defined by   $\delta_p({\cal B})=1$ if $p\in {\cal B}$, and equals $0$   otherwise. It is easy to verify that the map $p\mapsto \delta_p$ is a bijection between $X$ and all Dirac measures on $X$. Moreover, $p(x)$ is finitely satisfiable in $M$ if and only if   $\delta_p$  is finitely satisfiable in $M$. Indeed, one can easily check that,  for $(a_i)\subseteq A$,  $tp(a_i/{\cal U})\to p$ iff $\delta_{a_i}\to \delta_p$.
	\newline
	(iii) is well-known  (cf. \cite[Section~7.1]{Simon}). For $M$-finitely satisfiability, 
	 suppose that $\mu_r$ is the corresponding  regular Borel probability measure of $\mu$.
	It is easy to verify that $\mu$ is $M$-finitely satisfied (cf. Definition~\ref{fs-measure}) iff $\mu_r$ is $M$-finitely satisfied, that is, for every Borel set ${\cal B}\subseteq S_\phi({\cal U})$ with $\mu_r({\cal B})>0$ there is $a\in M$ such that $tp_\phi(a/{\cal U})\in {\cal B}$.
	\newline
	(iv) is a well-known result in Functional Analysis. Indeed, let $X=S_{\phi}({\cal U})$ and  $D(X)$ be the set of all Dirac measures on $X$; that is,  $D(X)=\{\delta_x: x\in X\}$, where $\delta_x({\cal B})=1$ if $x\in {\cal B}$, and  equals $0$   otherwise. Let $M(X)$ be the set of all regular  Borel probability measures on $X$.
	 By Proposition~437P of \cite{Fremlin4}, every $\delta_x\in D(X)$ is an extreme point of $M(X)$.  By Riesz Representation Theorem (\cite[Theorem~436J]{Fremlin4}),  there is a correspondence between regular  Borel probability measures $\mu\in M(X)$ and positive linear functionals $I_\mu$  on $C(X)$ such that $\|I_\mu\|=1$. (Therefore, $M(X)$ is a compact convex set of $C(X)^*$ with the weak* topology.) By the  Krein--Milman theorem (\cite[V, Theorem~7.4]{Conway}), the closed convex hull $\overline{\text{conv}}(D(X))$ of $D(X)$ is $M(X)$. Moreover, it is easy to verify  that the set of all $M$-finitely satisfiable  measures in $M(X)$ is a closed convex subset of $M(X)$, and   $M$-finitely satisfiable Dirac measures in $D(X)$ are its extreme points. Again, by the Krein--Milman theorem, the proof is completed. 
	\newline
	(v) follows from (i)--(iv), or Proposition~4.6 of \cite{Gannon}. Let $\mu\in{\frak M}_\phi(\cal U)$ be finitely satisfiable in $M$.
	There are two points. First, for any $\theta\in\Delta_\phi$, the closures of  functions of the forms  $\sum_1^n r_i\cdot\theta(a_i,\bar y)$ (where $\sum r_i=1$ and $r_i> 0$) and $\frac{1}{n}\sum_1^n  \theta(a_i,\bar y)$ 
	are the same.
	 Second, note that the functions of the form  $\sum_1^n r_i\cdot\phi(a_i,\bar y)$   alone are not enough to define a measure, because we need information about the Boolean combinations of the instances $\phi(x,b)$'s. Therefore, for every $\theta(x,\bar y)\in \Delta_\phi$ we need to have a function $f_\mu^\theta$ on $S_{\bar y}(M)$ as follows:  $f_\mu^\theta(q)=\mu(\theta(x,\bar b))$ where $\bar b\models q$. Now, by (iv) above or \cite[Proposition~4.6]{Gannon}, $f_\mu^\theta$ is in the closure of functions   $f_{a_i}^\theta$, $a_i\in M$ (where $f_{a_i}^\theta(q)=\sum_1^n r_i\cdot\phi(a_i,\bar y)$ for some/any $\bar b\models q$).
\end{proof}

\begin{Remark}   \label{remark global}
(i) Note that a global version of Fact~\ref{Dirac} is true
 for global types/measures instead of $\phi$-types/$\phi$-measures. The proof is just an adaptation of the argument of  Fact~\ref{Dirac}. Although the primary purpose of this article is to study local properties,  we will use the global version in Theorem~\ref{global angelic} below.
\newline
(ii) History. The correspondence between coheirs of $\phi$-types and suitable functions in (i) above was observed by Pillay \cite{Pillay}. The connection with Grothendieck's  result was first asserted  in \cite{Ben-Groth}. It seems that (in model theory) the relationship between Keisler $\phi$-measures and the convex hull of continuous functions was first observed  by Gannon \cite{Gannon}. The correspondence between Keisler measures and regular Borel probability measures on the type spaces is mentioned  in \cite[Section~7.1]{Simon}.
\end{Remark}

Let $A$ be a (small) set of the monster model $\cal U$. A set $F$ is called a {\em closed set over $A$} if it is the set of realizations in $\cal U$ of a partial type over $A$. The complements of closed sets over $A$ are called the {\em open sets over $A$.} The Borel sets over $A$ are the sets in the $\sigma$-algebra generated by the open sets over $A$. Alternatively, these   correspond  to the Borel sets of the relevant Stone space of complete types over $A$.

\begin{Definition}  \label{Borel definability}
{\em In the following, all types/measures are $A$-invariant.
	 \newline (i)(\cite{HP}) A global type $p(x)$ is called {\em Borel definable over $A$} if for any $L$-formula $\phi(x,y)$, the set $\{b\in{\cal U}: \phi(x,b)\in p(x)\}$ is a Borel set over $A$.
\newline 
(ii)(\cite{HP})  A global type $p(x)$ is called {\em strongly Borel definable over $A$}   (or {\em $DBSC$ definable over $A$}) if for any $L$-formula $\phi(x,y)$, the set $\{b\in{\cal U}: \phi(x,b)\in p(x)\}$ is a  {\em finite} Boolean combination of closed sets over $A$.\footnote{In response to the referee's question about the naming $DBSC$, we must mention that: (i) the notion of $DBSC$ function (on metric spaces) is well-known in functional analysis (cf. Fact~\ref{DBSC=strong}). (ii) This nomenclature makes it easier for people to search in analysis's results and use them for applications in model theory.}
\newline
 (iii)(\cite{HP}) A global (Keisler) measure $\mu(x)$ is called {\em Borel definable over $A$} if for any $L$-formula $\phi(x,y)$ and any closed subset $F\subseteq[0,1]$, the set $\{b\in{\cal U}: \mu(\phi(x,b))\in F\}$ is a Borel set over $A$.
\newline 
(iv) Let $\phi(x,y)$ be a formula and $p(x)$ a global $\phi$-type. Then {\em Borel definability} and {\em strong Borel definability} are defined similarly. In this case, the solutions correspond to Borel subsets of $S_{\tilde\phi}(A)$. The {\em Borel definability} of a global $\phi$-measure is defined similarly.  (Notice that for any $\theta\in\Delta_\phi$ the solutions correspond to Borel subsets of $S_{\tilde\theta}(A)$. Cf. the paragraph before Definition~\ref{fs-measure}.)
\newline 
(v) Let $B\subset\cal U$ be a small set of parameters such that every type over $A$ is realized in $B$. For non-global types and non-global $\phi$-types over $B$, the {\em Borel definability} and {\em strong Borel definability} over $A$ are defined similarly. For non-global measures and non-global $\phi$-measures, the {\em Borel definability} is defined similarly.}
\end{Definition}

\begin{Fact} \label{strong Borel=DBSC}  (a) Let $A$ be a (small) set and $p(x)$ a global type which is finitely satisfiable in $A$. For any  $L$-formula $\phi(x,y)$, define the function $p_\phi:S_{\tilde{\phi}}({A})\to\{0,1\}$, denoted by $p_\phi(q)=1$ iff $\phi(x,b)\in p$ for some (any) $b\models q$. Then the following are equivalent:
	\newline
(i)   $p(x)$ is strongly Borel definable over $A$.
\newline
(ii) For any $L$-formula $\phi(x,y)$, the function $p_\phi$ is  (generalized) $DBSC$. (Cf. Definitions~\ref{subclasses}, \ref{generalized subclasses} and Fact~\ref{DBSC=strong}.)

\medskip\noindent (b) The same holds for global $\phi$-types,  non-global types and non-global $\phi$-types.
\end{Fact}
\begin{proof}
	This   follows from Facts~\ref{DBSC=strong} and \ref{Dirac}(i). Indeed, if $S_y(A)$ is metrizable  this follows from  Fact~\ref{DBSC=strong}. If  $S_y(A)$ is not metrizable, this follows from  Definition~\ref{generalized subclasses}.
\end{proof}

We now give the suitable  {\em strong} notion  to measures. 

\begin{Definition}[Baire-1/2 definable measure]  \label{strong Borel measure}  {\em  In the following, all measures are $A$-invariant.
\newline 
(i) Let  $\mu(x)$  a global (Keisler) measure which is finitely satisfiable in $A$. For any  $L$-formula $\phi(x,y)$, define the function $\mu_\phi:S_{\tilde{\phi}}(A)\to[0,1]$, denoted by $\mu_\phi(q):=\mu(\phi(x,b))$  for some (any) $b\models q$.
  $\mu(x)$ is called {\em (generalized) Baire-1/2 definable over $A$} if for any $L$-formula $\phi(x,y)$, the function $\mu_\phi$ is (generalized) Baire-1/2 as in Definition~\ref{generalized subclasses}.
\newline 
(ii) For global $\phi$-measures, non-global measures and non-global $\phi$-measures, the {\em (generalized) Baire-1/2 definability} is defined similarly.  (Compare Definition~\ref{Borel definability}(iv),(v) and Fact~\ref{strong Borel=DBSC}.)}
\end{Definition}

In the next lemma we give a new characterization of $NIP$ for formulas. 
First, we need some definitions.
(Note the distinction  in form between the two definitions below. In fact, the part (i) of Definition~\ref{semi-NIP}  corresponds to $NIP$ in {\em classical} logic and the part (ii) corresponds to $NIP$ in {\em continuous} logic.)

\begin{Definition} \label{semi-NIP}
	{\em Let $X$ be a set and $(f_i)$ a sequence of $[0,1]$-valued functions on $X$.
	
\noindent
(i) We say the {\em independence property is
uniformly blocked on $(f_i)$}  if  there is a natural number $N$ such that for each $r<s$  there is a set
$E\subset\{1,\ldots,N\}$ such that  for  each $i_1<\cdots<i_{N}<\omega$, the following does not hold
$$\exists y\big(\bigwedge_{j\in E}f_{i_j}(y)\leq r~\wedge~\bigwedge_{j\in N\setminus E}f_{i_j}(y)\geq s\big).$$	
	
\noindent
(ii) We say the {\em independence property is
	semi-uniformly blocked on $(f_i)$}  if
for each $r<s$ there is a natural number $N_{r,s}$ and a set
$E\subset\{1,\ldots,N_{r,s}\}$ such that  for  each $i_1<\cdots<i_{N_{r,s}}<\omega$, the following does not hold
$$\exists y\big(\bigwedge_{j\in E}f_{i_j}(y)\leq r~\wedge~\bigwedge_{j\in N_{r,s}\setminus E}f_{i_j}(y)\geq s\big).$$ }
\end{Definition}  
\begin{Remark}
	(i): We remark that, for types in classical logic, two conditions (i),(ii) above are the same. Indeed,  notice that, as $f_i(y):=\phi(a_i,y)$ is $\{0,1\}$-valued, so $r=0$ and $s=1$. For measures in classical logic, the condition (i) is strictly stronger than (ii). 
	\newline
	 (ii): The key point in the above definition is that the   conditions (i),(ii) lead to modes of convergence  (i.e. $DBSC$ and Baire-1/2 convergence) that are stronger than the pointwise convergence. This creates a fundamental difference between the results of this article and \cite{Gannon},\cite{Gannon sequential}. (See also \cite{K-Morley} and \cite{K-generic}.) The difference between our work and Ben~Yaacov's result \cite[Lemma~5.4]{Ben} is that we do not assume that the sequence is indiscernible (cf. Lemma~\ref{bounded variation-continuous} below).
	 \newline
	 (iii): We note that the notion of dependence in \cite{Ben} is equivalent to the notion of sequential dependence in \cite{Gannon} using compactness theorem. On the other hand, Ben Yaacov's results \cite{Ben} is about $NIP$ theories (in continuous logic) but not $NIP$ formulas. Of course, as we will show in Lemma~\ref{bounded variation-continuous}, his results can be adapted to local case (formula-by-formula). Also, he considered indiscernible sequences \cite[Lemma~5.4]{Ben}, but the sequences are arbitrary in the following.
	 \newline
	 (iii): In  response to the question of why we use subsequence instead of sequence itself in Lemma~\ref{bounded variation-continuous} below, the reason is that we do not use indiscernible sequences but rather arbitrary sequences. Therefore, the sequence may not  converge, but a subsequence definitely converges.  (If we worked with indiscernible sequences, everything would  be simpler but we would not achieve our stronger goal.)\footnote{Therefore, to investigate any example for Lemma~\ref{bounded variation-continuous}, we must pay attention to finding an appropriate {\bf sub}sequence.}
\end{Remark}

\begin{Lemma} \label{bounded variation-continuous}  Let $T$ be a complete theory and $\phi(x,y)$  an $NIP$ formula. Suppose that $(c_n)$ is an infinite sequence (in the monster model of $T$) and
	 $(g_n)$ is a sequence of functions of the form $g_n=\sum_{i=1}^{k_n} r_i\cdot\phi(a_i,y)$ where $\sum r_i=1, r_i\in {\Bbb R}^+$, $n\in\Bbb N$ and $a_i\in\cal U$.  Then the following properties hold (are even  equivalent):
	 
	 \medskip
	\noindent (i) There is an infinite subsequence $(c_n')\subseteq(c_n)$ such that:
	
	(a) There is a real number $C$ such that
	for all $b$ in the monster model, $$\sum_{n<\omega}\big|\phi(c_{n+1}',b)-\phi(c_n',b)\big|\leq  C.\footnote{It is easy to show that this definition is exactly equivalent to having finite alternation number for $\phi$. Cf. \cite{S-invariant}, Section~2.1. Of course, in the present paper the sequence is no necessarily indiscernible.}$$
	
	(b)  The independence property is uniformly
	blocked on $(\phi(c_n',y))$.
	
		\noindent (ii) There is an infinite subsequence $(g_n')\subseteq(g_n)$ such that the independence property is semi-uniformly
		blocked on $(g_n')$. 

	\noindent (iii) Suppose that $(c_n')$ is the subsequence in (i) above, then  $(\phi(c_n',y))$ converges pointwise to a function $f$ which is   $DBSC$.	
	
	\noindent (iv)  Suppose that $(g_n')$ is the subsequence in (ii) above, then   $(g_n')$ converges pointwise to a function  $g$ which is  Baire-1/2.	
\end{Lemma}
\begin{proof}
	(i)(a) and (i)(b) follow from Proposition~2.14(iii) and Lemma~2.8 of \cite{K-Baire}. (In fact, they are equivalent.)
	
	\noindent (ii): We will  use some results of \cite{Ben} and \cite{Ben-Keisler}. Indeed, let $A=\{a_n:n<\omega\}$, $M^*$ a model such that every type over $A$ is realized in it, and $T_*=Th(M^*)$ the complete $L(M^*)$-theory of $M^*$.  Note that every $g_n$ can be extended to a global Keisler measure $\nu_n$. (Indeed, if $g_n=\sum r_i\phi(a_i,y)$, define  $\nu_n(x)=\sum r_i\text{tp}(a_i/M^*)$. That is, $\nu_n(\psi(x))=\sum r_i\psi(a_i)$    for all $L(M^*)$-formula $\psi(x)$.)
	 Note that $\nu_n$ is a measure on the Stone space $S(T_*)$.   Then, by Corollary 2.10 of \cite{Ben-Keisler}, every measure $\nu_n$ in $T_*$ corresponds to a type ${\bf p}_n$ in the randomization  $T_*^R$ of  $T_*$. In this case, for every $L(M^*)$-formula $\psi(x)$, $\nu_n(\{q:\psi(x)\in q\})=(\mu[[\psi(x)]])^{{\bf p}_n}$ where $\mu[[\psi(x)]]$ is the corresponding formula in $T_*^R$.  By Theorem~5.3 (or Theorem~4.1) of \cite{Ben}, as $\phi(x,y)$ is $NIP$ in $T$, its corresponding formula $\mu[[\phi(x,y)]]$ is dependent in $T_*^R$. Note that the Ben Yaacov's argument is essentially local (formula-by-formula), although the statement of Theorem~5.3 is global, it is easy to check that the local result holds too. Let ${\bf a}_n$ be a realization of ${\bf p}_n$.
	Now, as $T_*^R$ is a continuous theory and the formula $\mu[[\phi(x,y)]]$ is dependent in it,
	 there is a subsequence   $({\bf a}_n')\subseteq({\bf a}_n)$ such that  the independence property is semi-uniformly
	blocked on $(\mu[[\phi({\bf a}_n',y)]])$.\footnote{Note that the sequence  $({\bf a}_n)$  is not indiscernible, so we must find a suitable subsequence that has the desired property.} If not, by compactness (of continuous logic) and Ramsey's theorem, one can show that the formula $\mu[[\phi(x,y)]]$ has $IP$, a contradiction.  (The argument is similar to Theorem~4.3(iv) of \cite{K-classification}.)
	 This means that  the independence property is semi-uniformly
	 blocked on  the corresponding sequence $(\nu_n')$ of measures in $T$. Now, if we transfer this to $(g_n)$, then we have that the independence property is semi-uniformly
	 blocked on $(g_n')$. (Note that our terminology  is different from Ben~Yaacov's terminology \cite{Ben} as he called this notion  {\em uniform} dependence. Because the natural number $N_{r,s}$ depends on $r,s$, we believe that {\bf semi}-uniformly dependence is more suitable.)
	
	\noindent (iii) is Proposition~2.14(iv) in \cite{K-Baire}.
	
		\noindent (iv): First, notice that $(g_n')$ converges to a Baire-1 function, denoted by $g$. This is a well known result in functional analysis, due to Rosenthal \cite{Ros}.
		 Suppose, for a
		contradiction, that $g$ is not Baire-1/2.
	Equivalently, suppose  that 
		the limit $f$ of $(\mu[[\phi({\bf a}_n',y)]])$
		 is not Baire-1/2. 
		 
		 Now we want to use two notions of ranks/indexes for functions (i.e.  $K_m$ and $\alpha$) that were introduced in \cite[page~7]{HOR}.
		 For a real-valued function $f$ on a compact Hausdorff space $K$ and $r<s$, let $K_0(f;r,s)=K$ and $K_{\alpha+1}=\overline{K_\alpha\cap[f\leq r]}\cap\overline{K_\alpha\cap[f\geq s]}$. At the limit ordinal $\alpha$ we set $K_\alpha(f;r,s)=\bigcap_{\beta<\alpha}K_\beta(f;r,s)$.   We let $\alpha(f;r,s)=\inf\{\gamma<\omega_1~|K_\gamma(f;r,s)=\emptyset\}$ if $K_\gamma(f;r,s)=\emptyset$ for some $\gamma<\omega_1$ and let $\alpha(f;r,s)=\omega_1$ otherwise.
		 
		  As $f$ is not Baire-1/2, by the equivalence
		(1)~$\Leftrightarrow$~(5) of Proposition~2.3 of \cite{HOR}, there	are $r<s$ such that the ordinal index $\alpha(f;r,s)$ is	infinite; equivalently,  for all natural numbers $m$,	$K_m(f;r,s)\neq\emptyset$.   By Lemma~3.1 in \cite{HOR}, for all $m$ there are $r<r'<s'<s$ and a subsequence	$(\mu[[\phi({\bf a}_n'',y)]])_{n<\omega}$  of $(\mu[[\phi({\bf a}_n',y)]])_{n<\omega}$  so that   
		$$\exists y\big( \bigwedge_{i\in E}\phi({\bf a}_{n_i}'',y)\leqslant r' ~\wedge ~\bigwedge_{i\in F}\phi({\bf a}_{n_i}'',y)\geqslant s'\big)$$
		holds, for all disjoint subsets $E,F$ of $\{1,\ldots,m\}$ and $n_1<\cdots<n_m$. So, by the compactness theorem, $\mu[[\phi(x,y)]]$ has $IP$, a contradiction. This means that the limit of $(\nu_n')$  (or $(g_n')$) is Baire-1/2.
\end{proof}

In the following we give a result on definability of measures on countable models (sets), assuming $NIP$. In the next section, we give an alternative proof which is also a generalization to
 the general case (i.e. without assuming countability). In Section~5 we will give a result that is stronger in another respect.

Recall that $\Delta_\phi$ is the set of all Boolean combinations of the instances of $\phi(x,y)$ (with the variable $x$).

\begin{Theorem} \label{DBSC measure}
	The following are equivalent:
	\begin{itemize}
		\item[(i)] $\phi$ has $NIP$ with respect to $T$.
		
		\item[(ii)] For any $\mu(x)\in{\frak M}_\phi({\cal U})$ and any countable model $M$, if $\mu$ is finitely satisfiable in $M$, then $\mu$ is Baire-1/2 definable over $M$. 
		
				\item[(iii)]  For any $p(x)\in S_\phi({\cal U})$ and any countable model $M$, if $p$ is finitely satisfiable in $M$, then $p$ is  strongly Borel definable over $M$. 
	\end{itemize}
\end{Theorem}
\begin{proof}
	(i)~$\Longrightarrow$~(ii):   By Fact~\ref{Dirac}(v) above, for any $\theta(x)\in\Delta_\phi$, there is a function $f$ in the pointwise closure of the convex hull $\overline{\textrm{conv}}(A)$ of $A=\{\theta(a,y):S_{\tilde{\theta}}({\cal U})\to\{0,1\} |~ a\in M\}$ such that  $f(q)=\mu(\theta(x,b))$ for some (any) $b\models q$. Note that we can assume that the functions in $A$ (and so the function $f$) are on $S_{\tilde\theta}(M)$. Therefore, as $M$ is countable, $S_{\tilde\theta}(M)$ is Polish and these functions are defined on a Polish space.
	As $\theta$ is $NIP$, By Proposition~5J of \cite{BFT}, the convex hull of $A$ satisfies  any one of conditions of \cite[Theorem~4D]{BFT} iff $A$ does. As $\theta$ is $NIP$, and so $A$ satisfies the condition (viii) of  Theorem~4D, the convex hull of $A$ satisfies the condition (vi) of  Theorem~4D. This means that the convex hull $\textrm{conv}(A)$ of $A$ is angelic and so 
	 there is a sequence $(f_n)\in\textrm{conv}(A)$   such that $f_n\to f$; that is, 
	there is a sequence  $(f_n)$ of the form $f_n(y)=\sum_{i=1}^{k_n}r_i\cdot\theta(a_i,y)$ (where $a_i\in M$ and $\sum r_i=1$, $r_i\in{\Bbb R}^+$) such that $f_n\to f$ pointwise. Note that since  Proposition~5J of \cite{BFT} needs countability, one can consider rational convex hull of $A$; that is, the set of functions $\sum_{i=1}^{k}r_i\cdot\theta(a_i,y)$ where $r_i\in{\Bbb Q}\cap[0,1]$ and $\sum_{i=1}^{k}r_i=1$.
	  Again, as $\theta$ is $NIP$, by Lemma~\ref{bounded variation-continuous}(ii) above, there is a subsequence $(f_n')\subseteq(f_n)$ such that the independence property is semi-uniformly blocked on $(f_n')$, and consequently the sequence $(f_n')$ converges to a Baire-1/2 function, by Lemma~\ref{bounded variation-continuous}(iv). This means that $f$ is Baire-1/2 and so $\mu$ is Baire-1/2  definable.

(i)~$\Longrightarrow$~(iii) is similar to (i)~$\Longrightarrow$~(ii) and even easier.  By Fact~\ref{Dirac}(v) above, there is function $f$ in the pointwise closure of $A=\{\phi(a,y):S_{\tilde{\phi}}({\cal U})\to\{0,1\} |~ a\in M\}$ such that  $f(q)=1$ if $\phi(x,b)\in p$ for some/any $b\models q$, and  $f(q)=0$ otherwise.
 As $\phi$ is $NIP$, by the equivalence (vi)~$\Longleftrightarrow$~(viii) of  Theorem~4D of \cite{BFT}, there is a sequence $(\phi(a_n,y))$ in $A$ such that $\phi(a_n,y)\to f(y)$ pointwise. By Lemma~\ref{bounded variation-continuous}(iii), $f$ is (generalized) $DBSC$ or equivalently strongly Borel. (Cf. Fact~\ref{strong Borel=DBSC} above.)

	(iii)~$\Longrightarrow$~(i): By the equivalence (iv)~$\Longleftrightarrow$~(vi) of \cite[Theorem~2F]{BFT}, we need to show that for any countable model $M$, the family $\phi(a,y),a\in M$, is relatively compact in $M_r(X)$, where $X=S_{\tilde{\phi}}(M)$ and $M_r(X)$ is the space of all functions which are measurable with respect to all Radon measures on $X$ (cf. \cite[Definition~1A(l)]{BFT}). 
	Suppose that $M$ is a countable model, $(a_i)_{i\in I}\in M$ and $\phi(a_i,y)\to f(y)$ pointwise on $X=S_{\tilde{\phi}}(M)$. Note that $f(y)$ defines the type $p(x)=\lim_i tp(a_i/{\mathcal U})$. As $p$ is finitely satisfiable in $M$, by (iii), $f$ is Borel measurable. Therefore, by the condition (iv) of Theorem2F of \cite{BFT} holds, and so 
	the condition (vi) holds. This means, by compactness theorem and as $M$ is arbitrary, that $\phi(x,y)$ has $NIP$.

(ii)~$\Longrightarrow$~(i): As every type is a measure, the argument (iii)~$\Longrightarrow$~(i) above works well. 
\end{proof}
\begin{Remark}
	(i): For the direction (i)~$\Longrightarrow$~(iii) above, in Proposition~2.6 of \cite{HP}, a more general statement is proved for {\bf global} types, but not $\phi$ types. Their argument uses the Morley sequence of types which is not well-defined for $\phi$-types. Although, one can  extend finitely satisfiable  $\phi$-types to finitely satisfiable global types and then uses their argument. We will do this in  Corollary~\ref{generalized definable}.
	\newline(ii): Again we emphasize that the convergence in the above theorem (and Lemma~\ref{bounded variation-continuous}) is strictly stronger than the pointwise convergence in Lemma~4.7 of \cite{Gannon}. This matter is described in more detail in the articles \cite{K-Morley},\cite{K-generic}.
\end{Remark}

\begin{Remark} \label{remark global 2}
Note that Theorem~\ref{DBSC measure} can be proved for global types/measures instead of $\phi$-types/$\phi$-measures (see also Theorem~\ref{global angelic} below). That is, the following are equivalent   (cf. Remark~\ref{remark global} above):
\newline
(i) $T$ has $NIP$. 
\newline
(ii) For any global measure $\mu(x)\in{\frak M}({\cal U})$ and any countable model $M$, if $\mu$ is finitely satisfiable in $M$, then $\mu$ is Baire-1/2 definable over $M$. 
\newline
(iii) For any global type $p(x)\in S({\cal U})$ and any countable model $M$, if $p$ is finitely satisfiable in $M$, then $p$ is  strongly Borel definable over $M$. 
\end{Remark}

\section{Glivenko--Cantelli classes} 
The classical theorem of Glivenko and Cantelli is a generalization of the law of large numbers in probability theory. The families of sets/functions for which the consequence of the Glivenko--Cantelli theorem holds  are called Glivenko--Cantelli classes.\footnote{In a simple word, for a compact space $X$, a family $A$ of real-valued functions on $X$, and a Borel probability measure $\mu$ on $X$, we say that $A$ is a Glivenko--Cantelli class (with respect to $\mu$)  if the condition (ii) in Theorem~465M of \cite{Fremlin4} holds. See also Theorem~\ref{fim}(ii). (Of course, pay attention that in Theorem~\ref{fim} this property is stated for {\bf all}  Borel probability   measures.)} Vapnik and Chervonenkis gave a {\bf uniform} characterization of such classes in \cite{VC}. As mentioned before, Talagrand's characterization \cite{Talagrand} of   Glivenko--Cantelli classes has some advantages over \cite{VC}.\footnote{The difference between VC characterization and Talagrand's characterization is that:
in the latter, this property is given for a fixed Borel measure $\mu$, but in the former, this property is given for all Borel measures simultaneously. Therefore, a class of functions has VC property iff for every Borel probability measure $\mu$ this class has Talagrand's property with respect to $\mu$. In other words, VC property is stronger than Talagrand's property, and 
the latter is finer than the former because it can consider measures separately.} 
In this case,
the fundamental notion  has been  invented by David H. Fremlin, namely the Fremlin--Talagrand stability. From a logical point of view, this notion was first studied in \cite{K-amenable} in the framework of {\em  integral logic}.

\medskip
In the following, given a measure $\mu$ and $k\geqslant 1$, the symbol  $\mu^k$ stands for $k$-fold product of $\mu$ and  $\mu^*$ stands for  the outer measure of  $\mu$.

\begin{Definition}[Fremlin--Talagrand stability, \cite{Fremlin4}, 465B]  \label{Talagrand-stable}
Let $A\subseteq C(X)$ be a pointwise bounded family
of real-valued continuous functions  on $X$. Suppose that $\mu$
is a measure on $X$. We say that $A$ is {\em $\mu$-stable}, if
$A$ is a stable set of functions in the sense of Definition~465B
in \cite{Fremlin4}, that is, whenever $E\subseteq X$ is
measurable, $\mu(E)>0$ and $s<r$ in $\mathbb{R}$, there is some
$k\geqslant 1$ such that $(\mu^{2k})^*D_k(A, E,s,r)<(\mu E)^{2k}$
where
\begin{align*}
D_k(A, E,s,r) = \bigcup_{f\in A}\big\{w\in &
E^{2k}:f(w_{2i})\leqslant s, ~f(w_{2i+1})\geqslant r  \textrm{ for
} i<k\big\}.
\end{align*}
\end{Definition}

The following fact shows the similarity between this notion and the $NIP$ in model theory.

\begin{Fact}[Proposition~4 in \cite{Talagrand}]
	Let $X$ be a compact Hausdorff space, $A\subseteq C(X)$  a pointwise bounded family
	of real-valued continuous functions  on $X$, and $\mu$ (an extension of) a Borel measure on $X$.
	\newline
	(1): For any measurable set $E\subseteq X$, $s<r$ in $\mathbb{R}$ and $k\geqslant 1$, the set $D_k(A, E,s,r)$ is measurable, and so there is no need to use the outer measure.\footnote{Note that in Definition~465B in \cite{Fremlin4} Fremlin uses arbitrary functions but not measurable or continuous.}
	\newline
	(2): The following are equivalent: 
	\begin{itemize}
		\item[(i)] $A$ is $\mu$-stable.
		\item[(ii)]  There is no   measurable subset $E\subseteq X$ 	with  $\mu(E)>0$ and $s<r$  such that for each $k=\{1,\ldots k\}$ $$\mu^k\Big\{w\in E^k: \ \forall I \subseteq k \ \exists f\in A \ \bigwedge_{i\in I}f(w_i)\leqslant s \wedge \bigwedge_{i\notin I} f(w_i)\geqslant r  \Big\}=(\mu E)^k.$$
			\end{itemize}
\end{Fact}
\begin{proof} (1): First, notice that it is easy to verify that we can use $<$ and $>$  instead of $\leqslant$ and $\geqslant$. (Fremlin used $\leqslant$ and $\geqslant$ and Talagrand used $<$ and $>$.)\footnote{In this paper we interpret  functions with $\{0,1\}$-valued formulas, and so the difference between $<$ and $\leqslant$ (or $>$ and $\geqslant$) is negligible and not important.} Let $f\in C(X)$. Then the set $\{w\in 
	X^{2k}:f(w_{2i})< s, ~f(w_{2i+1})> r  \textrm{ for
	} i<k\big\}$ is open. Therefore,  $D_k(A, E,s,r)$ is a union of the sets of the form
$$E^{2k}\cap\{w\in 
X^{2k}:f(w_{2i})< s, ~f(w_{2i+1})> r  \textrm{ for
} i<k\big\}$$
and the set on the right side  is open. This means that  $D_k(A, E,s,r)$ is of the form $E^{2k}\cap O$ where $O$ is an open set. To summarize,  $(\mu^{2k})^*D_k(A, E,s,r)=\mu^{2k}D_k(A, E,s,r)$.

(2): This follows from Proposition~4 of \cite{Talagrand}. Indeed, notice that, by (1) above,  as all functions in $A$ are continuous, for each $k$, $s<r$, and $E\subset X$ measurable,  the set $D_k(A, E,s,r)$ is measurable. This means that $A$ satisfies the condition (M) in \cite[Proposition~4]{Talagrand}. Proposition~4 of \cite{Talagrand} is proved in page 843.
\end{proof}

\begin{Remark}
Recall from \cite[Corollary~3.15]{K3} that  a formula $\phi(x,y)$ has $NIP$ if and only if for any model $M$ and any Keisler $\phi$-measure $\mu(x)$ over $M$ the family $\phi(x,b),b\in M$ is $\mu$-stable.
\end{Remark}

The following theorem is a translation of  \cite[Theorem~465M]{Fremlin4} and asserts that a formula $\phi(x,y)$ has $NIP$ if and only if for any model $M$ and any Keisler $\phi$-measure $\mu(x)$ over $M$ the family $\phi(x,b),b\in M$ is a  Glivenko--Cantelli class with respect to $\mu$. It is easy to verify that Theorem~\ref{fim} is basically the VC-theorem, of course, it is interesting because it is a novel approach for proving the same theorem. In Remark~\ref{generalization}, we explain why our approach has some advantages over the original approach. 

In the rest of paper, for a formula $\phi(x,y)$, a type $p$ with variable $x$ and $|y|$-parameter $b$,
 $\phi(p,b)=1$ if $\phi(x,b)\in p$ and $\phi(p,b)=0$ otherwise.
\begin{Theorem}  \label{fim}
	The following are equivalent:
	\begin{itemize}
		\item[(i)] $\phi$ has $NIP$ with respect to $T$.
		\item[(ii)] For any model $M$ and any measure  $\mu\in{\frak M}_\phi(M)$, $\sup_{b\in M}|\frac{1}{k}\sum_1^k\phi(p_i,b)-\mu(\phi(x,b))|\to 0$ as $k\to\infty$ for almost every $(p_i)\in S_{\phi}(M)^{\Bbb N}$ with the product measure $\mu^{\Bbb N}$. 
	\end{itemize}	
\end{Theorem}
\begin{proof}
	(i)~$\Longrightarrow$~(ii): Let $M$ be a model of $T$, and $\mu(x)$ a Keisler measure over $M$.
	Let  $\bar\mu$ be the completion of $\mu$, that is, the smallest extension of $\mu$ which is complete. (See \cite[Thm 1.9]{Folland} for the definition of the completion of a measure.) Therefore, $\bar\mu$ is a Radon measure in the sense of \cite[411H(b)]{Fremlin4}.
	 By  Corollary~3.15 of \cite{K3}, as $\phi$ is $NIP$ (for theory $T$),  the set $A=\{\phi(x,b):b\in M\}$ is  $\bar\mu$-stable for every Radon measure on $S_{\phi}(M)$, in the sense of \cite[Definition~3.5]{K3}. (Note that as $\phi$ is $NIP$, we do not need to assume $M$ is $\aleph_0$-saturated and the argument of \cite[Corollary~3.15]{K3} works well.) By Theorem 465M of \cite{Fremlin4},  $\sup_{b\in  M}|\frac{1}{k}\sum_1^k\phi(p_i,b)-\bar\mu(\phi(x,b))|\to 0$ for almost every $(p_i)\in S_{\phi}(M)^{\Bbb N}$ with the product measure $\bar\mu^{\Bbb N}$.  Now, note that as $\phi(x,b)$ ($b\in M$) is continuous and $\bar\mu$ is the completion of the Borel measure $\mu$, we have $\mu(\phi(x,b))=\bar\mu(\phi(x,b))$ for all $b\in M$. Therefore,  $\sup_{b\in  M}|\frac{1}{k}\sum_1^k\phi(p_i,b)-\mu(\phi(x,b))|\to 0$ for almost every $(p_i)\in S_{\phi}(M)^{\Bbb N}$ with the product measure $\mu^{\Bbb N}$. (Indeed, recall that every $\bar\mu$-measurable set $E$ is of the form $F\cup N$ where $F$ is $\mu$-measurable and $N$ is a subset of a $\mu$-null set.)
	
	(ii)~$\Longrightarrow$~(i): Again, by Theorem 465M of \cite{Fremlin4}, for any Keisler measure $\mu(x)$ over $M$, the set $A$ above is $\bar\mu$-stable, and so $\phi$ is $NIP$, by Corollary~3.15 of \cite{K3}. (Alternatively, this follows from (v)~$\Longrightarrow$~(ii) of \cite[Theorem~2F]{BFT}, (ii)~$\Longrightarrow$~(i) of \cite[Theorem 465M]{Fremlin4}, and the fact that every function in the pointwise closure of a $\bar\mu$-stable set is $\bar\mu$-measurable   \cite[Proposition~465D(b)]{Fremlin4}.)
\end{proof}

\noindent {\em Explanation.} There are two points.
First: In Theorem~\ref{fim}, the transition of a regular Borel measure (i.e. Keisler measure) to its completion is essential. For this, the following example was suggested to us by  Gilles Godefroy: 
 Take $K$  a compact scattered space which contains a non-Borel subset $E$.\footnote{A space $X$ is scattered if every nonempty subset $A$ of it contains a point isolated in $A$. See \cite{Eng}, page 59.} Let $A$ be the unit ball of $C(K)$, considered as a set of continuous functions on $X$=the unit ball of $C(K)^*$ equipped with the weak*-topology. Then every sequence in $A$ has a pointwise convergent subsequence since $C(K)$ does not contain $\ell_1$ (or directly),  but the characteristic function of $E$ belongs to the pointwise  closure of $A$ and it is not Borel. This means that  the   relative sequential compactness (i.e. the equivalences of \cite[Thm~2F]{BFT}) does not imply Borel definability (of types/measures). Of course, if $X$ is metrizable, then it holds by \cite[Thm~3F]{BFT}. In general (i.e. non-metrizable space $X$), recall that  $NIP$ implies Borel definability, by \cite[Proposition~2.6]{HP}.  
 Second: There is a version of $\mu$-stability in \cite[465S]{Fremlin4}, which is called $R$-stable, so that 
  is more appropriate in some ways. Indeed, Talagrand \cite[Thm~9-4-2]{T84}  showed that a subset $A\subset C(X)$ is $R$-stable if and only if every function in the pointwise closure of $A$ is measurable. The only difference  between these versions is in the definition of product measures. Therefore, all results of $\mu$-stability hold for this (weaker) version, and the argument of Theorem~\ref{fim} can be simplified by using $R$-stability.
  
 \begin{Remark} \label{generalization}
 	(i) There are  other versions of  Theorem~\ref{fim}:
 	\newline
 	(a) For a formula $\phi(x,y)$ and model $M$, the formula $\tilde\phi$ is $NIP$ in $M$ if and only if for any $\mu\in{\frak M}_\phi(M)$, $\sup_{b\in M}|\frac{1}{k}\sum_1^k\phi(p_i,b)-\mu(\phi(x,b))|\to 0$ as $k\to\infty$ for almost every $(p_i)\in S_{\phi}(M)^{\Bbb N}$ with the product measure $\mu^{\Bbb N}$.\footnote{See \cite{KP} for the definition of $NIP$ in a model.}
 	\newline
 	(b)  For a formula $\phi(x,y)$, model $M$ and $\mu\in{\frak M}_\phi(M)$, the set $\{\phi(x,b):b\in M\}$ is $\mu$-stable if and only if $\sup_{b\in M}|\frac{1}{k}\sum_1^k\phi(p_i,b)-\mu(\phi(x,b))|\to 0$ as $k\to\infty$ for almost every $(p_i)\in S_{\phi}(M)^{\Bbb N}$ with the product measure $\mu^{\Bbb N}$.
 	\newline
 	(ii) In Theorem~\ref{fim}, if we consider $X_1,\ldots,X_m$  a finite collection of Borel over $M$ sets, then for  each $i\leq m$, the set $A_i=\{\phi(x,b)\cap X_i:b\in M\}$ is $\mu$-stable (for all Radon measures $\mu$). 
 \end{Remark}

One may demand a better result, that is, a result  for {\em global}  Keisler measures  (i.e. not only for the measures on the Boolean combinations of the instances of $\phi(x,y)$). We respond positively to his/her demand:

\begin{Corollary} \label{global}   Let   $T$ be a countable  theory.
The following are equivalent:
\begin{itemize}
	\item[(i)] $T$ is $NIP$.
	\item[(ii)] For any  model $M$ and any measure  $\mu\in{\frak M}(M)$, there is a set  $X\subseteq S(M)^{\Bbb N}$ of the full measure 1 such that for any formula $\theta(x,\bar y)$,
	 $$\sup_{{\bar b}\in M}|\frac{1}{k}\sum_1^k\theta(p_i,{\bar b})-\mu(\theta(x,{\bar b}))|\to 0$$ as $k\to\infty$ for all $(p_i)\in X$. 
\end{itemize}
\end{Corollary}
\begin{proof}
	Note that the set of all formulas (with the variable $x$) is {\bf countable}. Let $F=\{\theta_1(x),\theta_2(x),\ldots\}$ be an enumeration of them. Fixed 
 $\theta(x,\bar y)\in F$. By Theorem~\ref{fim}, it is easy to verify that the set $X_\theta$ of all $(p_i)\in S(M)^{\Bbb N}$  such that $\sup_{{\bar b}\in M}|\frac{1}{k}\sum_1^k\theta(p_i,{\bar b})-\mu(\theta(x,{\bar b}))|\to 0$ for all  $(p_i)\in X_\theta$ has the measure 1 with respect to $\mu^{\Bbb N}$. Let  $X=\bigcap_{n=1}^\infty X_{\theta_n}$.
 As every measure is continuous from above (cf. \cite[Theorem~1.8]{Folland}(d)), $\mu^{\Bbb N}(X)=\mu^{\Bbb N}(\bigcap_{n=1}^\infty X_{\theta_n})=\lim_k\mu^{\Bbb N}(\bigcap_{n=1}^k X_{\theta_n})=1$. 
This means that there is a set $X\subseteq S(M)^{\Bbb N}$ of the full measure 1 such that for any formula $\theta(x,\bar y)$, $\sup_{{\bar b}\in M}|\frac{1}{k}\sum_1^k\theta(p_i,{\bar b})-\mu(\theta(x,{\bar b}))|\to 0$ as $k\to\infty$ for all  $(p_i)\in X$. 
\end{proof}

The following remark will be used in Corollary~\ref{generalized definable} and Theorem~\ref{almost fam}. Although it is folklore, we provide the proof
for clarity and completeness.
\begin{Remark} \label{rem-density}
	 In Theorem~\ref{fim}, suppose that $N\prec M$ is a small model and $\mu$ is $N$-finitely satisfiable. Let $X$ be the set of all $N$-finitely satisfiable  $\phi$-types over $M$. Then it is easy to verify that
	for each $b\in M$,  $\mu(\phi(x,b))=\mu(\phi(x,b)\cap X)$. In fact, $\mu(X)=1$.\footnote{Assuming $NIP$ for theory, the same holds if  $\mu$ is $N$-invariant and $X$ is  the set of all $N$-invariant  $\phi$-types over $M$ (cf. \cite[Proposition 4.6]{HP}), but not in general.} 
\end{Remark}
\begin{proof}
Let $supp(\mu)$ be the support of $\mu$. As $\mu$ is $N$-finitely satisfiable, every type in $supp(\mu)$ is $N$-finitely satisfiable.
Recall from   \cite[Proposition~2.10]{Gannon-thesis} that $\mu(supp(\mu))=1$. Therefore $\mu(X)=1$. The same holds for $\phi$-measures/types.
\end{proof}

The following is a generalization of Theorem~\ref{DBSC measure}, since $M$ is not necessarily countable.
\begin{Corollary} \label{generalized definable}
	The following are equivalent:
\begin{itemize}
	\item[(i)] $\phi$ has $NIP$ (for theory $T$).
	
	\item[(ii)] For any $\mu(x)\in{\frak M}_\phi({\cal U})$ and any (not necessarily countable) model $M$, if $\mu$ is finitely satisfiable in $M$, then $\mu$ is (generalized) Baire-1/2 definable over $M$. (Cf. Definition~\ref{strong Borel measure}.) 
	
	\item[(iii)]  For any $p(x)\in S_\phi({\cal U})$ and any (not necessarily countable)  model $M$, if $p$ is finitely satisfiable in $M$, then $p$ is (generalized) $DBSC$ over $M$.  (Cf. Definition~\ref{Borel definability} and Fact~\ref{strong Borel=DBSC}.)
\end{itemize}
\end{Corollary}
\begin{proof}
	(i)~$\Longrightarrow$~(ii): By Theorem~\ref{fim}, for each $\epsilon>0$, there is $k$ such that 
	$\sup_{b\in \cal U}|\frac{1}{k}\sum_1^k\phi(p_i,b)-\mu(\phi(x,b))|\leq\epsilon$  for almost every $(p_i)\in S_{\phi}({\cal U})^{k}$ with the product measure $\mu^{k}$. By Remark~\ref{rem-density}, 
	note that $\mu(S_{\phi}({\cal U}))=\mu(X)$
	 where $X$ is the set of all all $M$-finitely satisfiable global $\phi$-types.
	 Recall from \cite[Proposition~2.6]{HP} that every invariant global type is strongly Borel definable. It is easy to check that every finitely satisfiable $\phi$-type (in a not necessarily $NIP$ theory) can be extended to a global finitely satisfiable type. Therefore, using an argument similar to Proposition~2.6 of \cite{HP},  one can show that every $p\in S_{\phi}({\cal U})$ is strongly Borel definable over $M$.
	 For each $k$, let $(p_1^k,\ldots,p_k^k)\in X^k$ such that $\sup_{b\in \cal U}|\frac{1}{k}\sum_1^k\phi(p_i^k,b)-\mu(\phi(x,b))|\leq\epsilon$. Let
	$f_k(y)=\frac{1}{k}\sum_1^kg_{i,k}(y)$ where $g_{i,k}$ defines $p_i^k$. Clearly, as the $g_{i,k}$'s are $DBSC$, $f_k$ is  $DBSC$. Therefore, the functions $f_k$'s ($k<\omega$) are strongly Borel and the sequence $(f_k)$ uniformly converges to a function $f$ which defines $\mu$. This means that $\mu$ is  (generalized) Baire-1/2 definable.

	(i)~$\Longrightarrow$~(iii) follows from Proposition~2.6 of \cite{HP}.
	
	(ii)~$\Longrightarrow$~(i) and (iii)~$\Longrightarrow$~(i) follow from the directions  (ii)~$\Longrightarrow$~(i) and (iii)~$\Longrightarrow$~(i) of Theorem~\ref{DBSC measure} above.
\end{proof}

\begin{Remark}
	(i) Assuming $NIP$ for the theory, the above result holds for invariant measures rather than coheirs. (See \cite{HP}, Lemma~4.8 and Corollary~4.9.)
	
	\noindent(ii) When $M$ is countable, the above corollary gives an alternative proof the direction (i)~$\Longrightarrow$~(ii) of Theorem~\ref{DBSC measure}.
	
		\noindent(iii) The interesting thing is that in the proofs of Corollary~\ref{generalized definable} and Theorem~\ref{DBSC measure}  we did not use the VC Theorem
		that was used on page 1025 of \cite{HP}.
\end{Remark}

The following is similar to an observation made by Kyle Gannon \cite[Corollary~3.10]{Gannon sequential}. In Remark~\ref{Karim-G}, we further discuss its connection to Gannon's observation.

In this paper, when we are discussing subalgebras in the context of the Stone--Weierstrass theorem the $r_i$ actually  can be negative, but in the context of a convex hull the must be non-negative. (Cf. Theorem~\ref{almost fam} and Proposition~\ref{SW-2} below.)
\begin{Theorem} \label{almost fam} Assume that $\phi(x,y)$ is $NIP$ and $\mu\in{\frak M}_\phi({\cal U})$ is finitely satisfiable in a (not necessarily countable) model $M$. Let $\nu$ be any local Keisler measure in ${\frak M}_{\tilde{\phi}}(M)$ and  $\epsilon>0$.

\noindent (i) There exists a $\nu$-measurable set $C\subseteq S_{\tilde{\phi}}(M)$ such that:

(1)  $\nu(C)>1-\epsilon$, and 

(2) there exist $\bar a_1,\ldots,\bar a_k\in M^{<\omega}$ such that for all $b$ (in the monster model) such that $tp_{\tilde{\phi}}(b/M)\in C$, we have: $$\Big|\sum_{i=1}^k r_i\cdot\theta_i(\bar a_i,b)-\mu(\phi(x,b))\Big|<\epsilon, \ \ \  (*)$$ where   $\theta_i$ is of the form $\bigwedge_j\phi(a_j,y)$  ($a_j\in \bar a_i$ and $k<\omega$ and $r_i\in{\Bbb R}$), or $\theta_i$ is the tautology  $\forall x(x=x)$. 

\noindent  (ii) Suppose moreover that $M$ is countable, then 
 
 (a)  there exist a measurable set $C$ with $\nu(C)>1-\epsilon$ and  $a_1,\ldots,a_n\in M$ such that for all $b$ (in the monster model) such that $tp_{\tilde{\phi}}(b/M)\in C$, we have: $|\sum_1^n r_i\cdot\phi(a_i,b)-\mu(\phi(x,b))|<\epsilon$    where $\sum_1^n r_i=1, r_i\in{\Bbb R}^+$.
 
  (b) Consequently,  there is a measurable subset $E\subseteq S_{\tilde{\phi}}(M)$ with $\nu(E)>1-\epsilon$ and Baire-1/2 function $f$ on $S_{\tilde{\phi}}(M)$ such that $\mu(\phi(x,b))=f(b)$ for all $b\in \cal U$, $tp_{\tilde{\phi}}(b/M)\in E$. \end{Theorem}
\begin{proof} (i): By Theorem~\ref{fim} and Remark~\ref{rem-density}, there are types $p_1,\ldots,p_n\in S_\phi({\cal U})$, which are finitely satisfiable in $M$ and $$\big|\frac{1}{n}\sum_1^n\phi(p_i,b)-\mu(\phi(x,b))\big|<\epsilon/2.$$ 

Recall from Fact~\ref{Dirac} that every $p_i$ corresponds to a function on $S_{\tilde{\phi}}(M)$, denoted by $p_i$ again,  which  defines it. (Notice that the such functions are Borel measurable by Corollary~\ref{generalized definable}(iii).) By Lusin's theorem (\cite[Theorem~7.10]{Folland}),
for each $i\leq n$, 
 there are a continuous functions $f_i$'s   on $S_{\tilde{\phi}}(M)$ and a $\nu$-measurable set $C_i$ such that $\nu(S_{\tilde{\phi}}(M)\setminus C_i)<\frac{\epsilon}{n}$ and $f_i(q)=p_i(q)$ for all $q\in C_i$. By the Stone--Weierstrass theorem (\cite[Thm.~4.45]{Folland}), for any $i\leq n$ there is a function $\psi_i(y)$ of the form $\sum_1^kr_i\cdot\theta_i(\bar a_i,y)$ (where  $\theta_i$ is of the form $\bigwedge_j\phi(a_j,y)$ and $a_j\in \bar a_i$, $r_i\in{\Bbb R}$) such that $\sup_{b\in S_{\tilde{\phi}}(M)}|f_i(q)-\psi_i(q)|<\epsilon/2$.\footnote{Notice that the set of such functions is an algebra and it separates points in $S_{\tilde{\phi}}(M)$. It also contains the constant function 1.} Let $C=\bigcap_1^n C_i$. Clearly, $\nu(C)>1-\epsilon$.
 So putting everything together we have
 \begin{align*}
 \Big|\frac{1}{n}\sum_1^n\psi_i(b)-\mu(\phi(x,b))\Big| & \leqslant \Big|\frac{1}{n}\sum_1^n\phi(p_i,b)-\mu(\phi(x,b))\Big| \\
 &  + \Big|\frac{1}{n}\sum_1^n\phi(p_i,b)-\frac{1}{n}\sum_1^n f_i(b)\Big|  \\
 &  +
 \Big|\frac{1}{n}\sum_1^n f_i(b)-\frac{1}{n}\sum_1^n\psi_i(b)\Big| <\epsilon/2+\epsilon/2 
 \end{align*}
 for all $b$ (in the monster model) such that $tp_{\tilde{\phi}}(b/M)\in C$. 
This proves  the desirable result. 

(ii): Suppose moreover that $M$ is countable. 

(a): Note that every $p_i$ ($i\leq n$) in  $(*)$ above is in the closure of the convex hull of the $\phi(a,y)$'s ($a\in M$). As $M$ is countable, by Proposition~5J of \cite{BFT}, any $p_i$ is the limit of a {\bf sequence} $(g_{n,i})$ (as $n\to\infty$) where $g_{n,i}$ is of the form $\sum r_{n,j}\cdot\phi(a_j,y)$ ($a_j\in M$ and $\sum r_{n,j}=1$, $r_{n,j}>0$). (Use the equivalence (vi)$\Longleftrightarrow$(viii) of Theorem~4D of \cite{BFT}.)   For each $i\leq n$, by Egorov's theorem, there exists a $\nu$-measurable set $C_i$ with $\nu(C_i)>1-\epsilon/n$ such that $(g_{n,i})$ {\bf uniformly}  converges to $p_i$ on $C_i$. Let $C=\bigcap_1^n C_i$. Again, putting everything together the desirable result follows.

(b): As $M$ is countable, by the previous paragraph, set $\bar{a}_n=(a_1^n,\ldots,a_k^n)\in M$ and $\nu$-measurable $C_n \subseteq S_{\tilde{\phi}}(M)$ such that $\nu(C_n)>1-\epsilon/2^n$ and $|\sum_1^k r_i\cdot\phi(a_i^n,b)-\mu(\phi(x,b))|<\frac{1}{n}$ for $b\in\cal U$, $tp_{\tilde{\phi}}(b/M)\in C_n$. Let $E=\bigcap_{n=1}^\infty C_n$. By Lemma~\ref{bounded variation-continuous}(ii), there is a {\bf sub}sequence of $\big(Av(\bar{a}_n)(\phi(x,y))\big)$ which is convergent to a Baire-1/2 function $f(y)$ on $S_{\tilde{\phi}}(M)$. Clearly, $\mu(\phi(x,b))=f(b)$ for all $b\in \cal U$ such that $tp_{\tilde{\phi}}(b/M)\in E$ and $\nu(E)>1-\sum_{n=1}^\infty \epsilon/2^n=1-\epsilon$. 
 \end{proof}
\begin{Remark} \label{Karim-G}  (i): Notice that in part (i) above the real numbers $r_i$ can be negative. We don't know if we can consider only positive numbers  or not.
	
\noindent (ii): As mentioned above, the  part (ii)(a) of Theorem~\ref{almost fam} (i.e. without Baire-1/2 definability) is similar to  \cite[Corollary~3.10]{Gannon sequential}. Indeed, as $T$ is countable in  \cite{Gannon sequential}, every sequentially approximated measure is finitely satisfiable in a {\bf countable} model (cf. Proposition~3.3(i) of \cite{Gannon sequential}). Of course, he has made no assumptions about the formula $\phi$, and we have made no assumptions about the measure $\mu$. (As mentioned in the paragraph before \cite[Corollary~3.10]{Gannon sequential}, the latter results is similar to {\em almost definability} of coheirs in \cite{K3}.)

\noindent (iii): The regularity of $\nu$ has been used for the general case (i.e., not necessarily countable models). However, in countable case, regularity is established automatically; that is, every Borel probability measure on a  metric space is regular. 

\end{Remark}

The following is another application of the Stone--Weierstrass theorem:
\begin{Proposition} \label{SW-2} Assume that $\phi(x,y)$ is a (not necessarily $NIP$) formula, $\mu\in{\frak M}_\phi({\cal U})$ and $M$ a model. If $\mu$  definable over $M$, then for any $\epsilon>0$  there exist $\bar a_1,\ldots,\bar a_k\in M^{<\omega}$ such that for all $b$ (in the monster model), we have: $$\Big|\sum_{i=1}^k r_i\cdot\theta_i(\bar a_i,b)-\mu(\phi(x,b))\Big|<\epsilon \ \text{for all } b\in{\cal U}, \ \ \boxtimes$$ where   $\theta_i$ is of the form $\bigwedge_j\phi(a_j,y)$  ($a_j\in \bar a_i$ and $k<\omega$ and $r_i\in{\Bbb R}$), or $\theta_i$ is the tautology $\forall x(x=x)$. 
\end{Proposition}
\begin{proof} As $\mu$ is definable over $M$, its corresponding function on $S_{\tilde{\phi}}(M)$ is continuous. By the Stone--Weierstrass theorem, it is in the uniform closure of the algebra generated by formulas $\phi(a,y),a\in M$. So $\boxtimes$ holds.
\end{proof}

\begin{Remark}
	(i) The	notion ``$NIP$ of $\phi(x,y)$ in a model"  was introduced in \cite{K3} and  more applications of it were presented in \cite{KP}. Some of the above results in the present  article-- not all them-- can be proved by the weaker assumption `$NIP$ in a model'. For this, one can use
	Fact~\ref{Dirac}  above and  Propositions~5I and 5J of \cite{BFT}.
	\newline
	(ii) As the types in {\em continuous} logic correspond to the measures in {\em classical} logic, all results of \cite[Section~3]{K3} can be translated for measures in classical logic. 
\end{Remark}

\subsection*{Thesis}
In \cite{K-Baire} we suggested a   hierarchy of unstable $NSOP$ theories using subclasses  of Baire~1 functions. The results of the present paper lead to a new hierarchy in $NIP$ theories. Indeed, let $X$ be a compact Polish space and $\cal C$ a class  of (real-valued) Baire~1 functions on $X$. Suppose moreover that $DBSC\subseteq{\cal C}\subseteq \text{Baire-}1/2$. (Recall from \cite{HOR} and \cite{KL} that there are many `certain' such classes.) We say that a complete theory $T$ is a $\cal C$-class (or is $\cal C$) if 

``for any formula $\phi(x,y)$ and any countable model $M$, every $M$-finitely satisfied Keisler $\phi$-measure $\mu$ over $M$ is $\cal C$-definable, that is, the functions that define $\mu$ are in $\cal C$."

\medskip
Several questions arise: 

\begin{Question}
	 Are there any unstable $NIP$ theories which are $DBSC$?  If so, Are there interesting such theories?  Is there any (model theoretic) characterization of $DBSC$ theories? Does it separate interesting theories? There are similar questions for each class ${\cal C}$ with $DBSC\subseteq{\cal C}\subseteq \text{Baire-}1/2$.
\end{Question}

\section{Angelic spaces of global types/measures}

In this section, we apply the work of Bourgain-Fremlin-Talagrand on Rosenthal compacta to give a short proof (and a refinement) of a    recent result of Gannon \cite{Gannon sequential}  (which was based on a result due to Simon \cite{S-invariant}).

\medskip
The key here is how to capture a ``global" type/measure with a {\em suitable} function. By ``global" we mean for {\em all} formulas, not just {\em one}.  Recall that Ben~Yaacov \cite{Ben-Groth} and Pillay \cite{Pillay}  had previously proposed a suitable function for the {\em local} types and {\em local} coheirs, respectively. We intend here to provide a {\em  suitable} generalization of this function that is appropriate for {\em global} types/measures. 

\medskip
First we recall some notations and notions. Given a model $M$ and types $p_1(x),\ldots,p_n(x)$ over $M$, the average measure of them, denoted by $\text{Av}(p_1,\ldots,p_n)$, is defined as follows:
$$\text{Av}(p_1,\ldots,p_n)(\theta(x)):=\frac{\big|\{i:~\theta(x)\in p_i, i\leq n\}\big|}{n}   \ \  \ \text{for all formula } \theta(x)\in L(M).$$
If $a_1,\ldots,a_n$ are elements in some model, $\text{Av}(a_1,\ldots,a_n):=\text{Av}(p_1,\ldots,p_n)$ where $a_i\models p_i$ ($i\leq n$). Similarly, given measures $\mu_1(x),\ldots,\mu_n(x)$ over $M$, the average measure of them, denoted by $\text{Av}(\mu_1,\ldots,\mu_n)$, is defined as follows:
$\text{Av}(\mu_1,\ldots,\mu_n)(\theta(x)):=\frac{1}{n}\sum_1^n\mu_i(\theta(x))$ for all formula   $\theta(x)\in L(M)$.

\medskip
Recall that for a topological space $X$ the subset $A$ of $X$ is 
 {\em  relatively countably compact} (in $X$) if every countable subset of $A$ has a cluster point in $X$. $A$ is   {\em relatively compact} if its closure is compact (in $X$).
 A regular Hausdorff space is {\em angelic} if (i) every relatively countably compact set is relatively compact; and (ii) the closure of a relatively compact set is precisely the limit of its sequences. 

\medskip 
The main theorem of \cite{BFT} asserts that for a Polish space $X$ the space $B_1(X)$ equipped with the topology of pointwise convergence is angelic.

\medskip 
A compact Hausdorff space is called {\em Rosenthal compactum} if it can be embedded in the space $B_1(X)$ of  Baire~1 functions on some Polish space $X$.

\medskip 
We are now ready to present the main result of this section.

\begin{Theorem} \label{global angelic}
	Let $T$ be a countable $NIP$ theory and $M$ a countable model of it.\footnote{An even weaker assumption can be considered. That is, if every formula is $NIP$ in $M$, then the arguments of (i), (ii) and (iv) work well. See  \cite{K3} for the definition of `$NIP$ in a model'.} Let $p(x)$ and $\mu(x)$ be a global type and a global measure which are finitely satisfiable in $M$, respectively. Then
	\begin{itemize}
		\item[(i)] 	(Simon) There is a sequence $(a_n)\in M$ such that $\lim_n tp(a_n/{\cal U})=p$.
		\item[(ii)] (Gannon) There is a sequence $(\bar{a}_n)\in M^{<\omega}$ such that $\lim_n \text{Av}(\bar{a}_n)=\mu$.
		\item[(iii)]	The limits in (i),(ii) above are $DBSC$ and Baire-1/2, respectively. That is, for every formula $\theta(x,y)$, the function $\lim_n\theta(a_n,y)$ (resp. $\lim_n \text{Av}(\bar a_n)(\theta(x,y))$) is $DBSC$ (resp. Baire-1/2) on $S_{\tilde\theta}(M)$.
			\item[(iv)] The space of global $M$-finitely satisfied types/measures is  a Rosenthal compactum.
	\end{itemize}
\end{Theorem}

\noindent
Before giving the proof let us remark:

\begin{Remark}
Note that	(iv) is a generalization of (i)/(ii) as follows: Suppose that  the space of global $M$-finitely satisfied  measures ${\frak M}^M_f({\cal U})$ embeds in $B_1(X)$ where $X$ is a Polish space. Therefore, by Theorem~3F of \cite{BFT},  ${\frak M}^M_f({\cal U})$ is angelic. (Recall that any  subspace of an angelic space is angelic.) On the other hand, ${\frak M}^M_f({\cal U})$ is the pointwise closure of the averages of global $M$-realized types. This proves (ii). Similarly for (i).
\end{Remark}

\begin{proof}[Proof of Theorem~\ref{global angelic}] 
	We want to store the information of each $M$-finitely satisfied type $p$ (resp. measure $\mu$)  by a function of the form $f_p:X\to[0,1]$ (resp. $f_\mu:X\to[0,1]$) where  $X$ is a Polish space.
	
	 Let $\theta_1(x,y_1),\theta_2(x,y_2),\ldots$ be an enumeration of all formulas with the variable $x$. For each $n$, let $X_n$ be the Stone space $S_{\tilde\theta_n}(M)$.  (Note that the $X_n$'s are Polish, since $M$ is countable.)
	 For each $n$, we adjoin an isolated point $q_n^*$ to $X_n$ and define $X_n^*:=X_n\cup\{q_n^*\}$ with the topological sum.\footnote{See \cite{Eng}, page 74.}  
	   Set $X=\prod_n X_n^*$ with the product topology. (As the product of a countable number of Polish spaces is  Polish, $X$ is Polish.)
	 We don't actually need the isolated points $\{q_1^*,q_2^*,\ldots\}$ but it is convenient to include them.

	Let $a\in M$ and $q\in S_{\tilde\theta}(M)$. We let $\theta(a,q)=1$ if for some (any)  $b\models q$ we have $\models\theta(a,b)$, and $\theta(a,q)=0$ in otherwise. For the isolated points $q_n^*$, we always assume that $\theta_n(a,q_n^*)=1$. In general, if $p(x)$  is a global $M$-finitely satisfied type, 
	 we let $\theta(p,q)=1$ if for some (any)  $b\models q$ we have $\theta(x,b)\in p$, and $\theta(p,q)=0$ in otherwise. For the isolated points $q_n^*$, we always assume that $\theta_n(p,q_n^*)=1$. Similarly, for a global $M$-finitely satisfied measure $\mu(x)$ and any isolated point $q_n^*$, we let $\mu(\theta_n(x,q_n^*))=1$.

	Let $p(x)\in S({\cal U})$ be finitely satisfiable in $M$.
	The following function $f_p$ stores the information of the type $p(x)$.  Define  $f_p:X\to[0,1]$ by $(q_n)\mapsto \sum_{n=1}^\infty \frac{1}{2^n}\theta_n(p,q_n)$. In general, if $\mu\in\frak{M}(\cal U)$ is a global $M$-finitely satisfiable measure, we define  $f_\mu:X\to[0,1]$ by $(q_n)\mapsto \sum_{n=1}^\infty \frac{1}{2^n}\mu(\theta_n(x,q_n))$.
	
	 \medskip\noindent
	We point out that    the information of the type $p(x)$ is stored in the function $f_p$. In fact, the map $p\mapsto f_p$ is a bijection from all global $M$-finitely satisfied types (with the variable $x$) onto all functions defined above. Furthermore,
	
	\vspace{4pt}
	\underline{Claim}: The map $p\mapsto f_p$ is bicontinuous. 
	
	\emph{Proof}:
	 We show that this map is continuous. The inverse is even easier. First notice that, $S({\cal U})$ is  equipped with the Stone topology and $[0,1]^X$ has the topology of pointwise convergence. Now consider the relative topologies in both cases.
	 
	 Let $(p_\alpha)$ be a net of global $M$-finitely satisfied types such that $(p_\alpha)$ converges to a type $p$. We will show that $f_{p_\alpha}\to f_p$. 	In the following we write $r\approx_\epsilon s$ if $|r-s|\leqslant \epsilon$. For each $\epsilon>0$ there is a natural number $N$ such that for any $(q_n)\in X$ we have:
	 \begin{align*}
	 \lim_\alpha f_{p_\alpha}(q_n)  & =\lim_\alpha\big(\sum_{n=1}^\infty\frac{1}{2^n}\theta_n(p_\alpha,q_n)\big)   =\lim_\alpha\big(\lim_k\sum_{n=1}^k\frac{1}{2^n}\theta_n(p_\alpha,q_n)\big) \\
	 &  \approx_\epsilon \lim_\alpha\big(\sum_{n=1}^N\frac{1}{2^n}\theta_n(p_\alpha,q_n)\big) 
	  =\sum_{n=1}^N\lim_\alpha\frac{1}{2^n}\theta_n(p_\alpha,q_n) \\
	 &  = \sum_{n=1}^N\frac{1}{2^n}\theta_n(p,q_n)  \approx_\epsilon
	  \sum_{n=1}^\infty\frac{1}{2^n}\theta_n(p,q_n) \\
	  & =f_p(q_n)
	 \end{align*}
Notice that for $\epsilon>0$ we can assume $N$ is the smallest natural number such that $\sum_{n=N}^\infty \frac{1}{2^n}<\epsilon$. As $\epsilon$ is arbitrary, the proof is completed.\hfill$\dashv_{\text{claim}}$ 
	
	\medskip
	
	 The same holds for global $M$-finitely satisfied measures.
	 
	 \medskip
	 
	 Therefore, every $M$-finitely satisfiable global type  corresponds to the limit of a net of functions in $\{f_a:a\in M\}$. (Cf. Fact~\ref{Dirac}).
	
	We emphasize that the functions $f_a$, $a\in M$, are continuous. Indeed, we define the following functions: for any $a\in M$ and $k\in \Bbb N$, we define a function $f_a^k:X\to[0,1]$ by $(q_n)\mapsto \sum_{n=1}^k \frac{1}{2^n}\theta_n(a,q_n)$. 
	Recall that projection functions and linear combinations of continuous functions are continuous. (In fact, the space of all continuous functions is a linear vector space.)
	Therefore the $f_a^k$'s are continuous. As $f_a^k\nearrow f_a$ uniformly and the uniform limit of a sequence of continuous functions is continuous, so $f_a$ is continuous.
	(We only used the $f_a^k$'s to prove the continuity of $f_a$, and we have no further work with the $f_a^k$'s in the rest of paper.)

	\medskip\noindent Let $A=\{f_a:a\in M, k\in{\Bbb N}\}$. 
	 We list our observations:
	\newline (1) $X$ is Polish.
	\newline (2) $f_a$ is continuous (for each $a\in M$).
	\newline (3)  The information of every global $M$-finitely satisfied type $p(x)$ is stored in a function $f_p$ in the closure $\overline A$ of $A$.  Furthermore, the map $p\mapsto f_p$ defined above is a homeomorphism, i.e, (i) it is one-to-one and onto, (ii) it is continuous and its inverse is so.
			
\vspace{4pt}
\underline{Claim}:  $A$ is relatively sequentially compact in $[0,1]^{X}$. That is, every sequence in $A$ has a pointwise convergent subsequence in $[0,1]^{X}$. 
 
\emph{Proof}:
Let $(f_{a_n})$  be a sequence in $A$. As $\theta_1,\theta_2,\ldots$ are $NIP$, by a diagonal argument, it is easy to verify that there is a convergent subsequence. Indeed, by Lemma~\ref{bounded variation-continuous}, let $(a_n^1)$ be a subsequence of $(a_n)$ such that the sequence $(\theta_1(a_n^1,y_1)$) converges as $n\to\infty$.  Similarly, by induction, let  $(a_n^i)$ be a subsequence of $(a_n^{i-1})$ such that the sequence $(\theta_i(a_n^i,y_i)$) converges as $n\to\infty$. Let $(c_n)$ be the diagonal of these sequences, i.e. $c_n=a_n^n$ for all $n$.\footnote{Alternatively, by an adaptation of \cite[Lemma~2.7]{S-invariant}, one can find a convergent  subsequence $(f_{c_n})$.}

It is easy to check that the subsequence $(f_{c_n})$ converges. Indeed, we show that for any $(q_i)\in X$ and any $\epsilon>0$, the sequence $(f_{c_n}(q_i))$ is Cauchy. Note that, for each $i$, there is $n_i$ such that  the sequence $(\theta_i(c_j,q_i):  n_i\leq j)$ is constant.
Let $m$ be the smallest natural number such that $\sum_{k=m}^\infty\frac{1}{2^k}<\epsilon$, and
let $N=\max\{n_i:i< m\}$. Therefore, $|f_{c_j}(q_i)-f_{c_{j'}}(q_i)|<\epsilon$ for all $j,j'> N$.\hfill$\dashv_{\text{claim}}$ 
	
\medskip	
By  the equivalence (i)~$\Longleftrightarrow$~(vi)  of \cite[Theorem 4D]{BFT}, every $f$ in the closure of $A$ is the limit of a {\bf sequence} in $A$.  By our translation (cf. Fact~\ref{Dirac} and Remark~\ref{remark global}),  this means that there is a sequence $(a_n)\in M$ such that $\lim_n tp(a_n/{\cal U})=p$. This proves~(i).

\medskip
(ii): Let $A$ be as above, $\text{Av}(A):=\{\frac{1}{n}(f_1+\cdots+f_n): f_i\in A,~ i\leq n\in{\Bbb N}\}$ and $\text{conv}(A):=\{\sum_1^nr_i\cdot f_i: f_i\in A,~n\in{\Bbb N},~r_i\in {\Bbb R}^+,~\sum_1^n r_i=1\}$. ($\text{conv}(A)$ is called the convex hull of $A$. If necessary, one can assume that $r_i\in{\Bbb Q}$.)
Clearly, the closures of $\text{Av}(A)$ and $\text{conv}(A)$ are the same, i.e. $\overline{\text{Av}}(A)=\overline{\text{conv}}(A)$.

Note that the information of $\text{Av}(a_1,\ldots,a_n)$ is stored in $\frac{1}{n}(f_{a_1}+\cdots+f_{a_n})$.

Again, by  \cite[Proposition~5J]{BFT},  as  $A$ is relatively sequentially compact, $\text{conv}(A)$ is relatively sequentially compact,  and so  every $h$ in the closure of $\text{conv}(A)$ is the limit of a {\bf sequence} $(g_n)$ in $\text{conv}(A)$.
In fact, we can find a sequence $(h_n)$ in $\text{Av}(A)$ such that $h_n\to h$. Indeed, suppose that $g_n=\sum_1^{k_n} r_i\cdot f_i$.
Define $h_n:=\sum_1^{k_n} s_i\cdot f_i$ where
 $|r_i-s_i|<\frac{1}{n\cdot k_n}$ and $s_i\in{\Bbb Q}$. Then $(h_n)\in \text{Av}(A)$ and $h_n\to h$.\footnote{For example, $\frac{3}{10}f_1+\frac{2}{10}f_2+\frac{1}{2}f_3=\frac{1}{10}(f_1+f_1+f_1+f_2+f_2+f_3+f_3+f_3+f_3+f_3)$.}

 By our translation (cf. Fact~\ref{Dirac} and Remark~\ref{remark global}), this proves~(ii).

\medskip
(iii) follows from Lemma~\ref{bounded variation-continuous}.

\medskip
(iv) follows from the arguments of (i) and  (ii). (Note that $f_{a_i}\to f_p$ if and only if $tp(a_i/{\cal U})\to p$. Furthermore, $\overline{A}\subseteq\overline{\text{conv}}(A)\subseteq B_1(X)$.)
\end{proof}

\medskip
At the end paper let us remark:

\begin{Remark}
The approach of the present paper allows us to make `easy' generalizations and to give more new results. We list some of them.
(i)~The above arguments  work  well in the framework of {\em  continuous logic} \cite{BBHU}. 
(ii)~The argument  can be adapted to global {\em  invariant} types/measures. That is, given a countable $NIP$ theory $T$ and countable model $M$, the space of global $M$-invariant types/measures is  a Rosenthal compactum. (This generalizes another result of Simon \cite[Theorem~0.1]{Simon-Rosenthal} that is about invariant $\phi$-types.)
 (iii)~The argument  leads us to another proof of the theorem that generically stable measures in $NIP$ theories (not necessarily countable) are also finitely approximated.
 (iv)~The argument leads us to the theorem that a countable theory is stable if and only if for {\em  any}  model $M$ the space of  global $M$-finitely satisfied types/measures is {\em Eberlein compactum}.\footnote{A compact space is an Eberlein compactum if 
 	it can be embedded in the space $C(X)$ of continuous functions on some compact Hausdorff space $X$.} (This generalizes the results in \cite{Ben-Groth} and   \cite{Pillay} in several ways. Indeed, recall from \cite[Thm~462B]{Fremlin4} that every Eberlein compactum is angelic.)
--- We will study (i)--(iv) in a future work.
\end{Remark}

\bigskip\noindent
{\bf Acknowledgements.}
I want to thanks David Fremlin, Gilles Godefroy and Pierre Simon
for   their helpful comments. I thank the anonymous referee for his detailed suggestions and corrections; they helped to improve significantly the exposition of this paper.
I would like to thank the Institute for Basic Sciences (IPM), Tehran, Iran. Research partially supported by IPM grant  no 1400030118.

\end{document}